\documentclass[12pt,reqno,twoside]{amsart}

\usepackage{amscd}
\usepackage{amsmath}
\usepackage{amssymb}
\usepackage{amsthm}
\usepackage{epsf}
\usepackage{verbatim}
\usepackage{graphicx}
\usepackage{parskip}
\usepackage{caption,hyperref}
\usepackage{color}
\input epsf.tex

\usepackage{etoolbox}
\makeatletter
\patchcmd{\@maketitle}
  {\ifx\@empty\@dedicatory}
  {\ifx\@empty\@date \else {\vskip0ex \centering\footnotesize\@date\par\vskip1ex}\fi
   \ifx\@empty\@dedicatory}
  {}{}
\patchcmd{\@adminfootnotes}
  {\ifx\@empty\@date\else \@footnotetext{\@setdate}\fi}
  {}{}{}
\makeatother

\theoremstyle{theorem} 
\newtheorem{theorem}{Theorem}[section]
\newtheorem*{theorem*}{Theorem}
\newtheorem{conjecture}[theorem]{Conjecture}
\newtheorem{lemma}[theorem]{Lemma}
\newtheorem{proposition}[theorem]{Proposition}
\newtheorem{corollary}[theorem]{Corollary}

\theoremstyle{definition} 

\theoremstyle{definition}

\theoremstyle{definition}
\newtheorem{remark}[theorem]{Remark}
\theoremstyle{definition}
\newtheorem{example}[theorem]{Example}
\theoremstyle{definition}
\newtheorem{definition}[theorem]{Definition}
\newtheorem*{thm1}{Theorem \ref{thm:sametriangle}}
\newtheorem*{thm2}{Theorem \ref{billiardsiet}}
\newtheorem*{thm3}{Theorem \ref{thm:ayiet}}
\newtheorem*{thm4}{Theorem \ref{thm:tree}}

\usepackage{chngcntr}
\counterwithout{figure}{section} 

\newcommand{\R}{\mathbb{R}}

\newcommand{\Z}{\mathbb{Z}}

\newcommand{\Q}{\mathbb{Q}}

\newcommand{\ve}{\varepsilon}

\newif\ifdraft\drafttrue

\begin{document}

\title[Tiling Billards on Triangle Tilings, and Interval Exchange Transformations]{Tiling Billards on Triangle Tilings, and\\ Interval Exchange Transformations}

\author[]{Paul Baird-Smith, Diana Davis, Elijah Fromm, Sumun Iyer}

\date{\today}

\begin{abstract}
We consider the dynamics of light rays in triangle tilings where triangles are transparent and adjacent triangles have equal but opposite indices of refraction. We find that the behavior of a trajectory on a triangle tiling is described by an orientation-reversing three-interval exchange transformation on the circle, and that the behavior of all the trajectories on a given triangle tiling is described by a polygon exchange transformation. We show that, for a particular choice of triangle tiling, certain trajectories approach the Rauzy fractal, under rescaling.
\end{abstract}

\maketitle

\section{Introduction}

A \emph{triangle tiling} is a planar tiling by congruent copies of a triangle so that the tiling is a grid of parallelograms with parallel diagonals. Equivalently, it is the image of the equilateral triangle tiling under an affine transformation. We describe the behavior of \emph{tiling billiards} trajectories on triangle tilings, using the refraction rule that when a trajectory ray hits an edge of the tiling, it is reflected across that edge (Figure \ref{refractionrule}(a)). As usual, if the ray hits a vertex, the subsequent trajectory is undefined. 

We prove the following powerful result about trajectories on triangle tilings:

\begin{thm1}
If a trajectory passes through the same triangle twice, the trajectory is in the same position each time.
\end{thm1}

Thus every non-escaping trajectory is periodic, and a periodic trajectory forms a simple, closed curve.

\begin{figure}[!h]
\begin{center}
\includegraphics[width=0.9\textwidth]{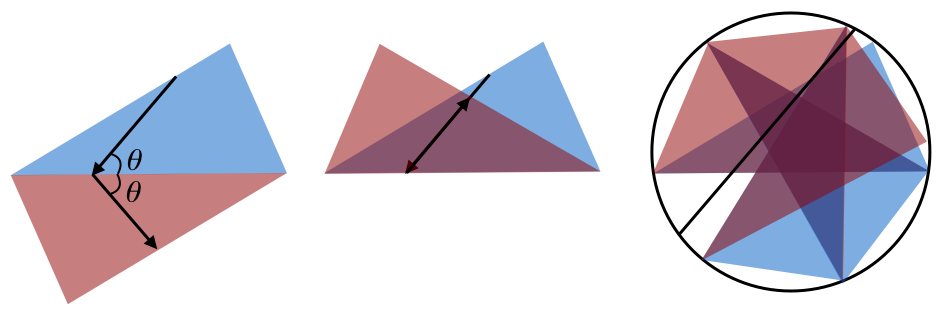}
\caption{(a) A tiling billiards trajectory refracts across an edge of a triangle tiling. (b) When folded across an edge, the pieces of trajectory lie on the same line. (c) In the folded position, all of the triangles are inscribed in the same circle, and the folded trajectory lies on a chord. \label{refractionrule}}
\end{center}
\end{figure}

As a crucial tool, we \emph{fold} trajectories along every edge of the tiling, which maps pieces of the trajectory onto the same line (Figure \ref{refractionrule}(b)) and inscribes every triangle into a single circle (Figure \ref{refractionrule}(c)). This allows us to turn the two-dimensional tiling billiards system into a one-dimensional system, where the circle and chord are fixed and we keep track of how one vertex of the triangle moves around the circle. Its movement can be described by an orientation-reversing 3-interval exchange transformation (IET):

\begin{thm2}
The movement of a trajectory whose chord subtends angle $\tau$ in a triangle tiling with angles $\alpha,\beta,\gamma$ is described by the orbit of a point under the \emph{tiling billiards IET}, which is a circle exchange transformation: the unit circle is cut into intervals of length $2\alpha, 2\beta$, $2\gamma$, each interval is flipped in place, and the circle is rotated by $\tau$.
\end{thm2}

This equivalence allows us to see the orbit of a point in an IET as a path in the plane. Additionally, the square of the orientation-reversing, 3-interval tiling billiards circle exchange transformation is an orientation-preserving circle exchange transformation with at most 6 intervals. For a particular choice of angles, it is the \emph{Arnoux-Yoccoz IET} (Definition~\ref{ayiet}):

\begin{thm3}
Let $a\approx 0.54$ be the real solution to $x+x^2+x^3=1$. For the \emph{Arnoux-Yoccoz triangle tiling} with angles $\alpha = \pi(1-a)/{2}$, $\beta = \pi(1-a^2)/{2}$, $\gamma = \pi(1-a^3)/{2}$, the square of the tiling billiards IET for the trajectory through the circumcenter is the Arnoux-Yoccoz IET scaled onto the interval $[0, 2\pi]$.
\end{thm3}

Trajectories on the Arnoux-Yoccoz triangle tiling approach the \emph{Rauzy fractal}. The closer the trajectory chord is to the circumcenter of the inscribing circle, the larger and more fractal-like the trajectories are. See Appendix \ref{rauzyfrac} for the first 12 such trajectories.

Finally, we have observed no periodic trajectories on triangle tilings that enclose a triangle; they only seem to enclose vertices and edges of the tiling, which thus form a \emph{tree}. We conjecture that this holds for all periodic trajectories on triangle tilings, and prove it in the obtuse case:

\begin{thm4}
A periodic trajectory on an obtuse triangle tiling encloses a path.
\end{thm4}

\subsection{Background on tiling billiards}

Physicists have recently discovered materials with a negative index of refraction, called \emph{metamaterials} \cite{shelby,smith}. We imagine obtaining a standard material (such as plastic or glass) with a positive index of refraction, and a metamaterial with the opposite index of refraction. For two-colorable tilings, such as the blue and red triangle tilings pictured in this paper, we can make the blue triangles out of the standard material, and the red triangles out of the metamaterial. Because the indices of refraction are equal and opposite, a beam of light passing through this tiling would take the path of a tiling billiards trajectory, with angles preserved across the edge as in Figure \ref{refractionrule}(a).

Sergei Tabachnikov originally suggested the system of tiling billiards, motivated by the unpublished work \cite{MF} of the physicists Mascarenhas and Fluegel, who investigated tiling billiards on the simplest tilings.
The second author, DiPietro, Rustad and St Laurent \cite{icerm} explored tiling billiards on three classes of tilings, all formed by lines: a division of the plane by finitely many lines, triangle tilings as defined above, and the \emph{trihexagonal} tiling where an equilateral triangle and a regular hexagon meet along each edge. That paper made many observations and conjectures about the behavior of trajectories on triangle tilings, which motivated our study in the present paper; our Corollaries \ref{nospiralnodense} and \ref{cor:4n2obtuse} resolve conjectures from \cite{icerm}.

The fact that a given trajectory crosses a given (triangular) tile at most once is in contrast to the behavior of trajectories on other tilings. The second author and Hooper studied the trihexagonal tiling, and showed that for that tiling, \emph{dense} behavior is general:

\begin{theorem*} \cite[Theorem 1.1]{trihex}
For almost every initial point and direction, a trajectory with this initial position and direction is dense in all of the plane except for a periodic family of triangular open sets.  
\end{theorem*}

Our system models a (two-colorable) triangle tiling tiled by materials with equal and opposite indices of refraction, whose ratio is the \emph{refraction coefficient} $-1$. Glendinning \cite{glendinning} studies billiards on the (two-colorable) square tiling, where the two colors of tiles have different \emph{positive} indices of refraction, where the refraction coefficient is greater than 1. It turns out that, as in our case, the billiard system under this rule can be described by an IET:

\begin{theorem*} \cite[\S 1]{glendinning}
On the square tiling, if the refraction coefficient is greater than $\sqrt 2$, the dynamics of trajectories can be described by an interval exchange transformation.
\end{theorem*}

Arnaldo Nogueira was the first to study ``IETs with flips,'' where the orientation is reversed on at least one interval. His study was motivated by flows on non-orientable surfaces; see \cite[Figure 1]{nonergodic}, which can be interpreted as modified tiling billiards on a rectangle.

For an extensive survey comparing the dynamics of tiling billiards on various tilings with the dynamics of inner billiards and outer billiards, see \cite[\S 1]{icerm}.

\subsection{Acknowledgments}
This project began at the SMALL REU at Williams College, funded by National Science Foundation REU Grant DMS - 1347804. Additional funding came from the Williams College Finnerty Fund, the Bronfman Science Center of Williams College, and the Clare Boothe Luce Program of the Henry Luce Foundation.

The conferences  ``Cycles on Moduli Spaces, Geometric Invariant Theory, and Dynamics'' at ICERM in August 2016, ``Teichm\"uller Space, Polygonal Billiard, Interval Exchanges'' at CIRM in February 2017\footnote{DD's lecture on this paper is available at \href{https://www.youtube.com/watch?v=CO5bVlRWmow}{https://www.youtube.com/watch?v=CO5bVlRWmow}} , and ``Teichm\"uller Dynamics, Mapping Class Groups and Applications'' at Institut Fourier in June 2018 provided rich research environments.

We thank Pat Hooper, Pascal Hubert, Curt McMullen, Arnaldo Nogueira, Olga Paris-Romaskevich, Rich Schwartz, Sasha Skripchenko, Dylan Thurston, and Barak Weiss for productive conversations about this project.

\section{Tiling billiards on triangle tilings, and folding}\label{sec:basic}

\subsection{Triangle tilings}

\begin{definition}\label{basiclabelling}\label{refracrule}

Label the angles of the tiling triangle $\alpha$, $\beta$, and $\gamma$ in order of increasing size, and orient the triangle so they are ordered counterclockwise. The sides opposite these angles are $A$, $B$, and $C$ respectively. Scale the triangle so that it is circumscribed by the unit circle. 
Rotate the tiling so that $C$ is horizontal. A \emph{positively-oriented} triangle has $C$ at the bottom, and a \emph{negatively-oriented} triangle has $C$ at the top. In this paper, the blue triangles are positively oriented and the red triangles are negatively oriented.
\end{definition}

Many billiard trajectories on triangle tilings exhibit periodicity, which comes in two types:

\begin{definition}\label{def:dp}
We call a trajectory \emph{periodic} with period $n$ if it repeats after intersecting with $n$ edges. We call a trajectory \emph{drift-periodic} if it is periodic up to a ``drift'', i.e. if the trajectory is invariant under translation by a nonzero vector. See Figure \ref{fig:per-dp}.
\end{definition}

\begin{figure}[!h]
\begin{center}
\includegraphics[width=0.6\textwidth]{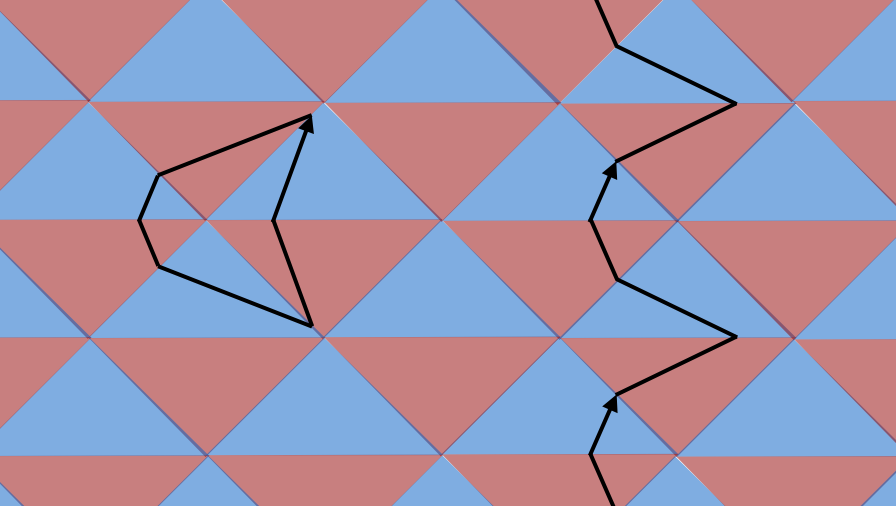}
\caption{A periodic trajectory with period 6 (left) and a drift-periodic trajectory with period 4 (right) on the isosceles right triangle tiling \label{fig:per-dp}}
\end{center}
\end{figure}

\subsection{Folding}

An important tool in studying inner billiards on polygons is \emph{unfolding}, in which a new copy of the table is created at each bounce. For tiling billiards, we use the complementary action of \emph{folding}, where we fold up the plane at each refraction.

\begin{definition}\label{foldtwotriangles}
Triangle $T_1$ is \emph{folded onto} adjacent triangle $T_2$ by reflecting triangle $T_1$ across the edge shared by triangles $T_1$ and $T_2$ (Figure \ref{refractionrule}(b)).
A connected collection of triangles from the triangle tiling is in the \emph{folded position} if it is folded on every edge between adjacent triangles in the collection.
\end{definition}

\begin{lemma}\label{folding}
Let $T_1$ and $T_2$ be two adjacent triangles of a tiling with a trajectory passing through their shared edge. When $T_1$ is folded onto $T_2$, the piece of the trajectory in $T_1$ and the piece of the trajectory in $T_2$ lie on the same line.
\end{lemma}

\begin{proof}
This follows immediately from the refraction rule; see Figure \ref{refractionrule}(b).
\end{proof}

\begin{lemma}\label{circumcenters}
Let $T_1$ and $T_2$ be two adjacent triangles of a triangle tiling. When $T_1$ is folded onto $T_2$, the two triangles share a circumcenter.
\end{lemma}

\begin{proof}
When $T_1$ is folded onto $T_2$ over their shared side $S$, $T_1$ and $T_2$ share the perpendicular bisector to $S$. As $T_1$ and $T_2$ are congruent and $S$ corresponds to the same side in both triangles, the distance along this bisector to the circumcenter is the same for both triangles, so the folded triangles share a circumcenter (see Figure \ref{foldedtri}).
\end{proof}

\begin{figure}[!h]
\begin{center}
\includegraphics[width=0.8\textwidth]{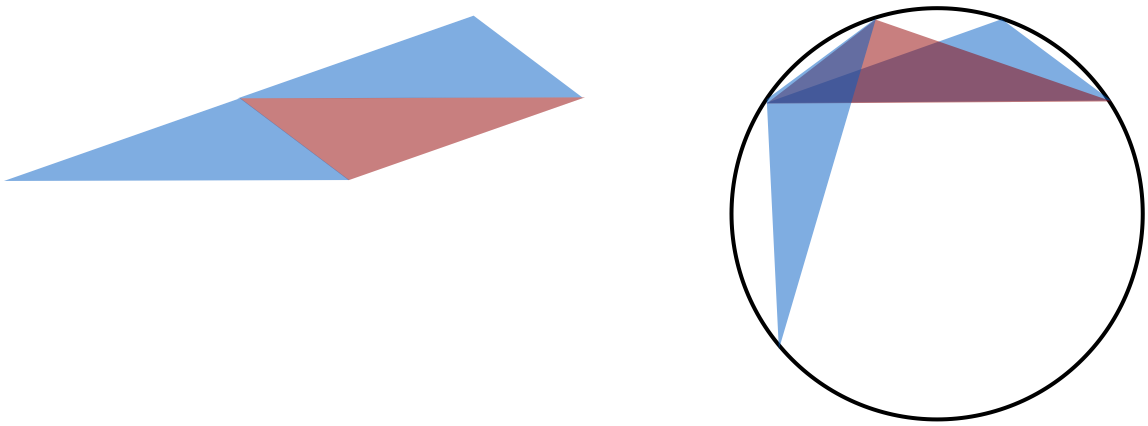}
\caption{After folding the three triangles on the left, all of the triangles share a circumcenter (Lemma \ref{circumcenters}), and the blue triangles are rotations of one another about the circumcenter (Lemma \ref{lem:rotations}). \label{foldedtri}}
\end{center}
\end{figure}

\begin{lemma}\label{circdist}
Given a trajectory on a triangle tiling, the perpendicular distance between the circumcenter of given triangle and a piece of the trajectory in that triangle (or its extension) is the same for all triangles hit by the trajectory.
\end{lemma}

\begin{proof}
In the folded position, the triangles share a circumcenter (Lemma \ref{circumcenters}), and the pieces of trajectory all lie on a single line (Lemma \ref{folding}). The local arguments in both Lemmas extend to all of the triangles that the trajectory crosses, by considering one adjacent pair of triangles at a time. Thus, the distance from the circumcenter to the line is the same for all triangles in the folded position, and consequently in the unfolded position as well.
\end{proof}

The distance from the trajectory to the circumcenter of the triangle is therefore an \emph{invariant} of the system. This allows us to transform the two-dimensional tiling billiards system into a one-dimensional system of interval exchanges on a circle, which we do in \S\ref{sec:circle}.

\begin{corollary}\label{allincircle}
The entire triangle tiling can be folded along tiling edges. Then:
\begin{enumerate}
\item every triangle in the plane lies in one circumscribing circle, and
\item for a tiling billiards trajectory on the tiling, all of the pieces of trajectory lie on a single chord of the circle.
\end{enumerate}
\end{corollary}

\begin{proof}
There is a path of adjacent triangles between any two triangles in the tiled plane, so we can fold up the plane by folding along every edge crossed in such a path.
\begin{enumerate}
\item When we fold along this path, by Lemma \ref{circumcenters}, the two triangles land in the same circumscribing circle. Since this applies to any pair of adjacent triangles, the entire plane can be folded, one triangle at a time, into a single circle (see Figure \ref{foldedtri}).
\item Since every fold takes pieces of trajectory onto the same line, this extends globally and all pieces of trajectory lie on the same line, which is a chord of the circle.
\end{enumerate}
\end{proof}

\begin{lemma}\label{lem:rotations}
In the folded state, all of the triangles of a given orientation are rotations of each other about the circumcenter.
\end{lemma}

\begin{proof}
Each fold is a reflection, and the composition of two reflections is a rotation. Since the triangles are inscribed in the same circle (Lemma \ref{circumcenters}), this rotation is about the center of that circle.
\end{proof}

\begin{remark}
The reader may wish to cut out and fold up the example tiling billiards trajectory in the Appendix, to tangibly understand the results of this section.

\begin{figure}[!h]
\begin{center}
\includegraphics[height=110pt]{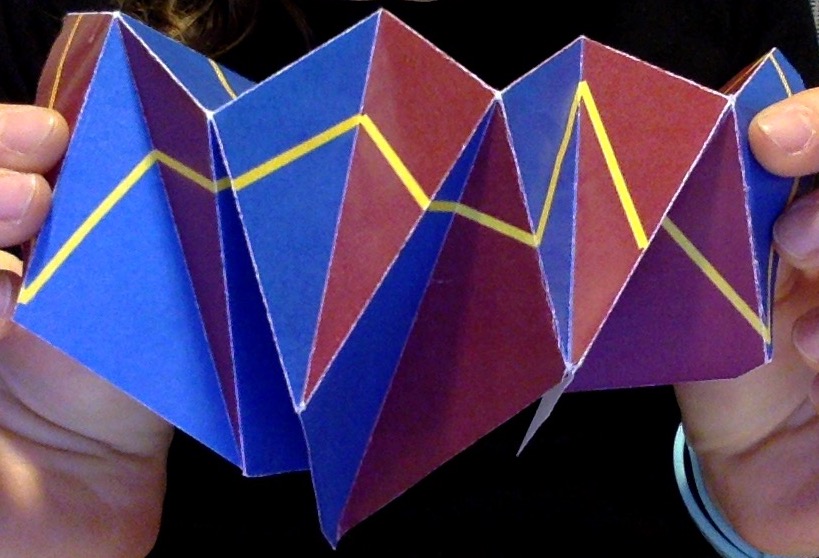} \ \
\includegraphics[height=110pt]{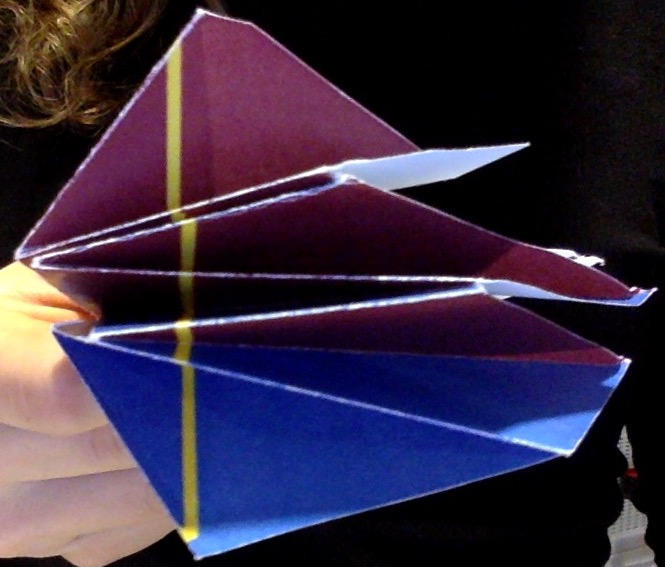} \ \
\includegraphics[height=110pt]{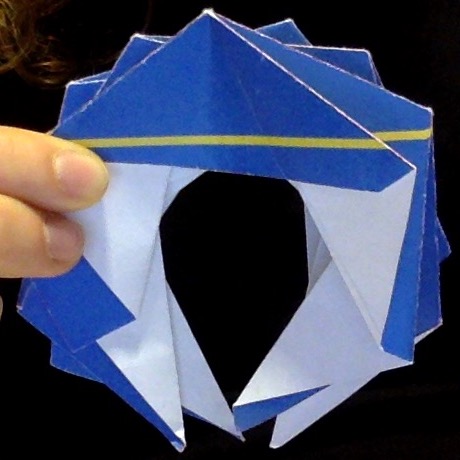} \ \
\caption{Folding up the example trajectory in the Appendix: (a) Folding along every edge of the tiling. (b) In the folded position, all the pieces of trajectory lie on the same line, and (c) all of the triangles are inscribed in the same circle, with the blue triangles on one side (shown) and red on the other. \label{fig:hands}}
\end{center}
\end{figure}

If the tiling triangle is obtuse, it is possible to actually fold the plane along every edge so that the triangles are all inscribed in the same circle, with an intact piece of paper (see Figure \ref{fig:hands}). The second author and Pat Hooper discovered this by folding up trajectories, but the construction is well known in the origami community. In mathematics, it is known as a \emph{Schwarz lantern}; Hermann Schwarz gave it as a counterexample to the assertion that the surface area of a triangulated approximation to a surface limits to the area of the surface.

If the tiling triangle is acute, it is not always possible to physically fold up the intact plane, since parts of the paper may run into each other. We can get around this problem by cutting out the trajectory along the edges of the tiling that the trajectory does \emph{not} intersect, at which point it is possible to fold along all of the edges of the tiling that the trajectory crosses, and see the triangles in their inscribing circle.
\end{remark}

\subsection{A trajectory crosses a triangle at most once}

Previous work on tiling billiards conjectured that trajectories on triangle tilings never exhibit dense behavior \cite[Conjecture 4.14]{icerm}. In Theorem \ref{sametrajectory}, we prove this by proving a much stronger result: A given trajectory only crosses a given triangle at most once.

First, we prove some preliminary results about how the trajectory moves around on the tiling. Given a starting triangle, any path returning to that triangle must cross an even number of edges, since every time an edge is crossed, the orientation of the triangle changes. Thus, we may consider edge crossings in pairs.

\begin{definition}\label{def:moves}
Recall that the angles of the triangles in the tiling are labeled $\alpha$, $\beta$, and $\gamma$ counter-clockwise in non-decreasing order. From a given positively-oriented triangle, there are six positively-oriented triangles that a trajectory can reach by crossing two edges (see Figure \ref{fig:moves}). Crossing edge $C$ and then $B$, the trajectory cuts past angle $\alpha$ in the counter-clockwise direction around the vertex adjacent to angle $\alpha$ in the negatively-oriented triangle, so we call this a $\alpha$-move. The other five moves are defined similarly, in Table \ref{table:moves}.
\end{definition}

\begin{figure}[!h]
\begin{center}
\includegraphics[width=0.8\textwidth]{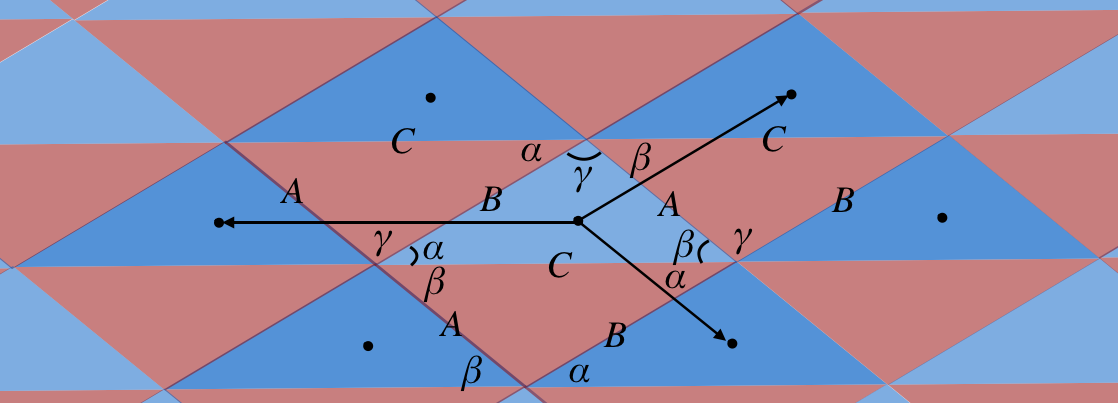}
\caption{From a given positively-oriented triangle (center), there are six adjacent positively-oriented triangles (dark blue), which are reached via the six two-edge moves described in Definition \ref{def:moves}. \label{fig:moves}}
\end{center}
\end{figure}

\begin{table}[!h]
  \begin{center}
\begin{tabular}{  r | c | r | c}
move name & edges crossed  & direction & translation  \\ \hline
$\alpha$-move & $CB$ & counter-clockwise & $\vec{A}$ \\
$-\alpha$-move & $BC$ & clockwise & $-\vec{A}$ \\
$\beta$-move & $AC$ & counter-clockwise & $\vec{B}$\\
$-\beta$-move & $CA$ & clockwise & $-\vec{B}$\\
$\gamma$-move & $BA$ & counter-clockwise & $\vec{C}$\\
$-\gamma$-move & $AB$ & clockwise & $-\vec{C}$
\end{tabular}
\caption{Two-edge moves to adjacent triangles of the same orientation as the original. The last column shows the translation vector to a reference point, in terms of the edge vectors of the triangle. Refer to Figure \ref{fig:moves} for an illustration.
\label{table:moves}}
  \end{center}
\end{table}

\begin{definition}
For a given path of even length, the integers \emph{$n_\alpha$, $n_\beta$, and $n_\gamma$} give the net number of $\alpha$, $\beta$, and $\gamma$ moves, respectively, in that path. More precisely, for $i = \alpha,\beta,\gamma$,
$$n_i= \text{(number of $i$ moves)}- \text{(number of $-i$ moves)}.$$
\end{definition}

\begin{lemma}\label{nx}
For any path in a triangle tiling, the path starts and ends in the same triangle if and only if $n_\alpha=n_\beta=n_\gamma$.
\end{lemma}

\begin{proof}
Suppose we have a path that starts and ends in the same triangle. Note that every such path is of even length, because triangle orientations alternate, so there is a whole number of moves. Each of the six moves always produces the same translation, relative to a reference point in each triangle of a given orientation (the last column of Table \ref{table:moves}). So, we can rearrange the moves in the sequence by their type ($\alpha$, $\beta$, or $\gamma$) and cancel out moves with their opposites, without changing the ending triangle of the sequence. Our new rearranged path, if any moves remain, forms a triangle of $\alpha$, $\beta$, and $\gamma$ moves in the positive or negative directions. In order to form a triangle, the remaining moves must either all be clockwise or all counterclockwise, and the numbers of tiling triangles on each edge of this larger triangle must be equal. This implies $n_\alpha=n_\beta=n_\gamma$ for the original sequence.

If, for a given path, we have $n_\alpha=n_\beta=n_\gamma$, then the path can be rearranged into a triangle as above, and we find that it ends on the triangle it started from.
\end{proof}

\begin{definition}
The \emph{turning} of a trajectory is the angle deviation from its initial direction. The \emph{total turning} is the net amount of turning for some number of refractions.
\end{definition}

\begin{lemma}\label{lem:turning}
The turning from an $\alpha$-move is $+2\alpha$, the turning from a $-\alpha$-move is $-2\alpha$, and similarly for $\beta$ and $\gamma$. The total turning for a sequence of two-edge moves is $2(\alpha n_\alpha + \beta n_\beta + \gamma n_\gamma)$.
\end{lemma}

\begin{proof}
In an $\alpha$-move, the trajectory crosses edges $C$ and $B$, which meet at a vertex with angle $\alpha$. Folding the tiling across these two edges (the top part of Figure \ref{fig:move_angles}) takes the three pieces of trajectory onto the same line (Corollary \ref{allincircle}). The overlapping parts of the plane each have angle $\alpha$, so unfolding the triangles back into the flat plane performs a rotation of $2\alpha$ (the bottom part of Figure \ref{fig:move_angles}). Proofs of the other five cases are analogous. The total turning from $\alpha$-moves is $2\alpha n_\alpha$, and similarly for $\beta$ and $\gamma$.
\end{proof}

\begin{figure}[!h]
\begin{center}
\includegraphics[width=0.9\textwidth]{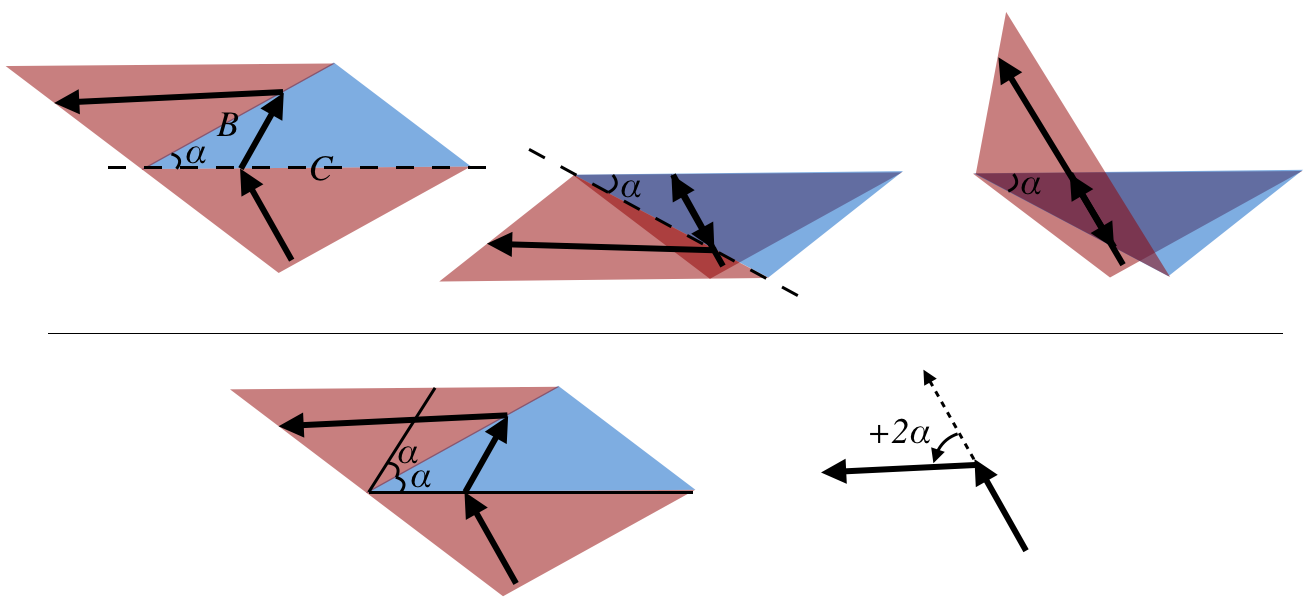}
\caption{A trajectory making a $+\alpha$-move makes a counter-clockwise turn of $+2\alpha$. Top: We fold the trajectory across the edges $C$ and $B$ crossed by the trajectory, along each of the dashed lines. Bottom: The two folds result in a turning of $+2\alpha$ in the trajectory's direction.  \label{fig:move_angles}}
\end{center}
\end{figure}

\begin{proposition}\label{no-dp}
If $\alpha,\beta$ and $\gamma$ are \emph{not} rationally related, then there is \emph{no} drift-periodic trajectory.
\end{proposition}

\begin{proof}
By Lemma \ref{nx}, a trajectory returns to its starting triangle if and only if  we have $n_\alpha = n_\beta = n_\gamma$, which occurs when the trajectory returns to the original triangle. So, if a trajectory is drift-periodic, then at least one of those equalities fails to hold. For a trajectory to be drift-periodic, it must be in the same position in different triangles. That is, at some point, the total turning $2(\alpha n_\alpha + \beta n_\beta + \gamma n_\gamma)$ must be a multiple of $2\pi$. However, since not all of $n_\alpha,n_\beta,$ and $n_\gamma$ are equal, this implies a rational relationship between $\alpha$, $\beta$, and $\gamma$.
\end{proof}

We use the following simple, somewhat surprising result about trajectories on triangle tilings to obtain numerous powerful results in Theorem \ref{sametrajectory}.

\begin{theorem}\label{thm:sametriangle}
If a trajectory passes through the same triangle twice, 
the trajectory is in the same position each time.
\end{theorem}

\begin{proof}
By Lemma \ref{lem:turning}, the total turning $\theta$ of the trajectory for a series of two-edge crossings is given by $\theta=2(\alpha n_\alpha+\beta n_\beta + \gamma n_\gamma)$.
Lemma \ref{nx} tells us that for a path that returns to the starting triangle, $n_\alpha=n_\beta=n_\gamma$, so $\theta=2(n_\alpha)(\alpha+\beta+\gamma)=2\pi n_\alpha$. So, the trajectory has completed an integer number of rotations, and it is in the same position as it started in.
\end{proof}

\begin{theorem}\label{sametrajectory}
The following hold true for any trajectory on a triangle tiling:
\begin{enumerate}
\item Every non-escaping trajectory is periodic. \label{nonescapeperiodic}
\item Every non-periodic trajectory escapes.\label{nonperiodicescapes}
\item Every periodic trajectory forms a simple, closed curve.\label{simpleclosedcurve}
\end{enumerate}
\end{theorem}

\begin{proof}
Each of these follows from Theorem \ref{thm:sametriangle}:

\begin{enumerate}
\item Suppose a trajectory is non-escaping. Then, the trajectory must cross some triangle of the tiling at least twice. By Theorem \ref{thm:sametriangle}, these two crossings are the same.
Thus, the trajectory is periodic.

\item This is the contrapositive of part $(1)$.

\item A periodic trajectory forms a closed curve, so we only need to show that a periodic trajectory cannot intersect itself. By Theorem \ref{thm:sametriangle}, if a trajectory intersects a given triangle of the tiling, it does so in exactly one line segment. Thus, a given trajectory cannot intersect itself within any triangle.
If the trajectory intersects itself at an edge of the tiling, it must contain two distinct line segments within each of the two triangles that meet at that edge, a contradiction. Therefore, the trajectory cannot intersect itself and so forms a simple, closed curve.
\end{enumerate}
\end{proof}

\begin{corollary} \label{nospiralnodense}
We prove Conjectures 4.13 and 4.14 of \cite{icerm}:
\begin{enumerate}
\item Trajectories on triangle tilings never spiral.
\item Trajectories on triangle tilings never fill a region of the plane densely.\label{nodense}
\end{enumerate}
\end{corollary}


\begin{proof} These follow directly from Theorem \ref{thm:sametriangle}:
\begin{enumerate}
\item In the inward direction of the spiraling, once the trajectory gets to the center of the spiral, it is ``stuck,'' as it cannot cross any triangle more than once, so spiraling cannot occur.
\item Every trajectory crosses a given triangle in at most a single line segment. Since the triangles have nonzero area, trajectories are not dense in the plane.
\end{enumerate}
\end{proof}

\begin{corollary}
If a trajectory passes through two triangles that share an edge, the pieces of the trajectory in each of the two triangles are reflections of each other across the shared edge, even if the shared edge is not hit by the trajectory.
\end{corollary}

See Figure \ref{nicetree} for a path exhibiting many examples of this property.

\begin{proof}
Suppose the trajectory passes through triangles $T_1, T_2, \dots , T_n$, across edges $e_1, e_2, \dots,$ $e_{n-1}$, where triangle $T_1$ is adjacent to triangle $T_n$, sharing edge $e_r$. By Theorem \ref{thm:sametriangle}, triangle $T_n$ must be in the same place relative to the trajectory from $T_1$ when folded across edges $e_1$ through $e_{n-1}$ as it is when folded across edge $e_r$. That is, the trajectory that results in triangle $T_n$ from refraction through triangles $T_2$ through $T_{n-1}$ is the same as that observed when reflecting the trajectory across edge $e_r$ from triangle $T_1$.
\end{proof}

\section{From trajectories to interval and polygon exchange transformations}\label{sec:circle}

Tiling billiards appears to be a two-dimensional system, with trajectories refracting around the plane. However, because the distance from a trajectory to the circumcenter of the triangle is an invariant of each trajectory, we can reduce tiling billiards to a one-dimensional system. This is in contrast to inner (standard) billiards, which has no such invariant, and is thus in some sense intrinsically ``harder'' than tiling billiards.

In the folded position, all of the triangles are circumscribed into one circle, with the pieces of trajectory on a single chord, so we can reduce the system to keeping track of an identified vertex of positively-oriented triangles. This allows us to transform our two-dimensional problem about trajectories on tilings into a one-dimensional problem about rotations on the circle, which we can describe using interval exchange transformations (IETs).

\subsection{Parameters in the circle, to set up the IET}\label{sec:circle_setup}

In this technical section, we derive the equations describing how the folded triangles containing a given trajectory rotate around their circumscribing circle, and how the parameters change when the tiling is folded.

\begin{definition}
\label{xtdef}
We choose a tiling triangle, circumscribed by the unit circle, containing an oriented segment of the trajectory. We extend this oriented segment to a chord of the circle.
The counterclockwise angle from the back end of the chord to the front end is denoted by $\tau$, and the counterclockwise angle from the vertex with angle $\alpha$ to the front end of the chord is denoted by $X$ (Figure \ref{xtau}). The \emph{position} of a trajectory within a given triangle is given by an ordered pair $(X, \tau)$.
\end{definition}

\begin{figure}[!h]
\includegraphics[width=0.35\textwidth]{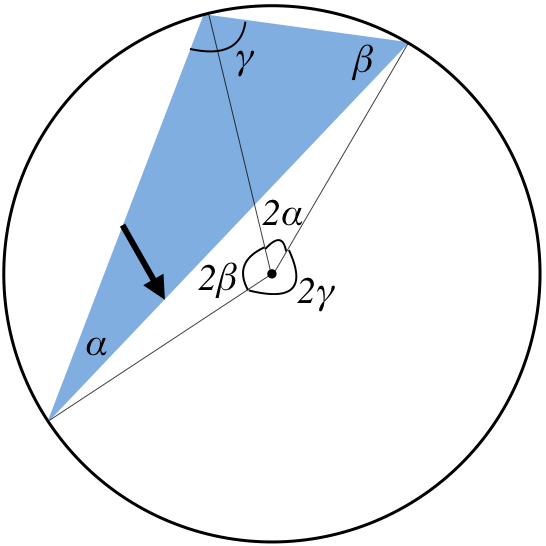} \qquad
\includegraphics[width=0.35\textwidth]{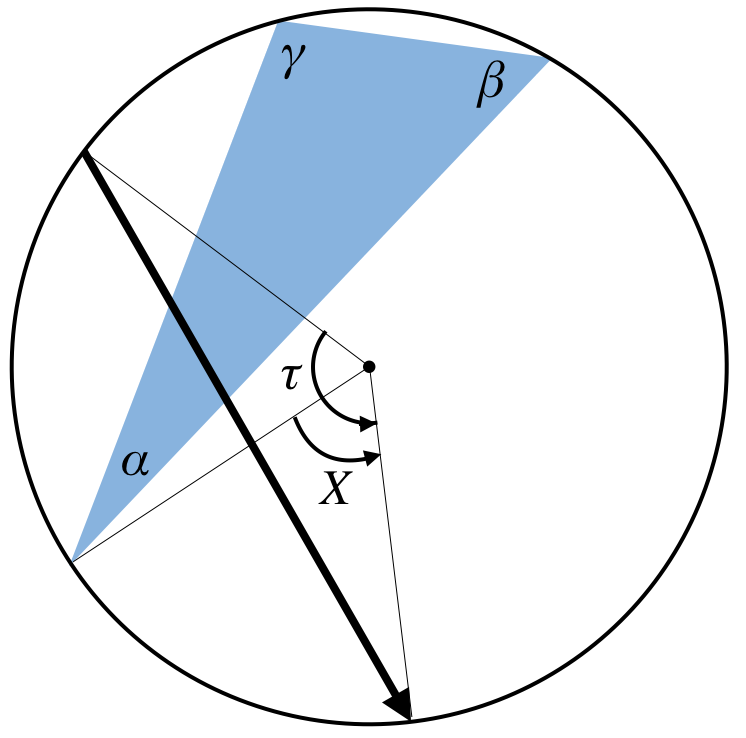}
\caption{(a) We start with a piece of trajectory in a triangle. (b) We reflect the triangle if necessary so that the angles $\alpha$, $\beta$, $\gamma$ are counter-clockwise. Then $X$ is the counter-clockwise angle from the vertex of angle $\alpha$ to the front end of the chord, and $\tau$ is the angle subtended by the trajectory (Definition \ref{xtdef}). \label{xtau}}
\end{figure}

\begin{lemma}\label{taufixed}
When a trajectory crosses an edge, $\tau$ does not change.
\end{lemma}

\begin{proof}
Consider two consecutive triangles passed through by the trajectory, and let $\tau$ be the counterclockwise angle subtended by the trajectory's chord in the first triangle. When the second triangle is folded onto the first along the shared edge, the segments of trajectory in each triangle align, with opposite orientations (Corollary \ref{allincircle}). In this configuration, the counterclockwise angle subtended by the chord of the trajectory in the second triangle is now $2\pi-\tau$, as it is the same chord as before, only flipped. After flipping the second triangle back to its correct orientation, the subtended angle is $\tau$ (see Figure \ref{singlehit}).
\end{proof}

\begin{figure}
\begin{center}
\includegraphics[width=0.31\textwidth]{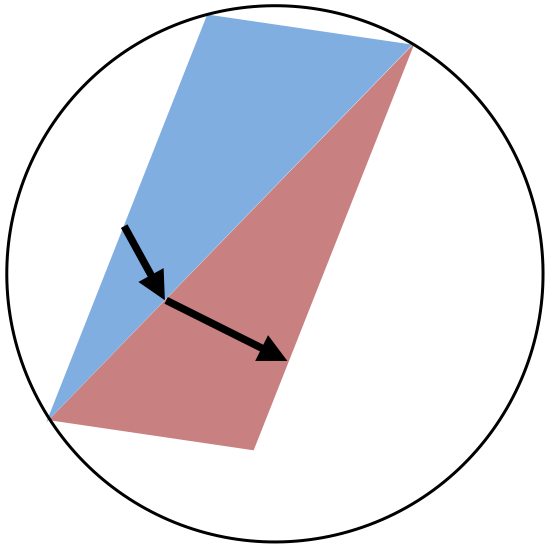} \ \
\includegraphics[width=0.31\textwidth]{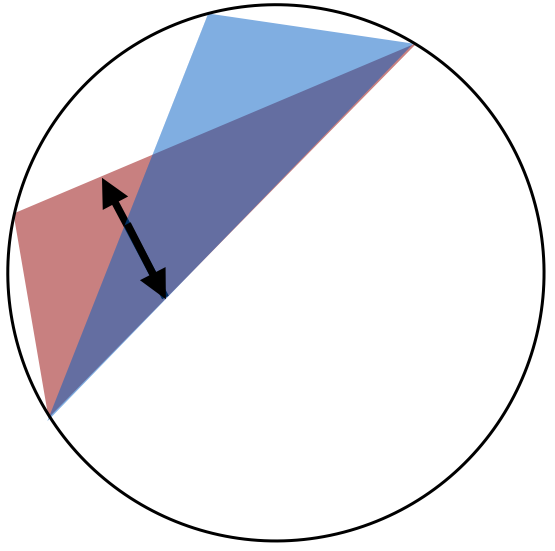} \ \
\includegraphics[width=0.31\textwidth]{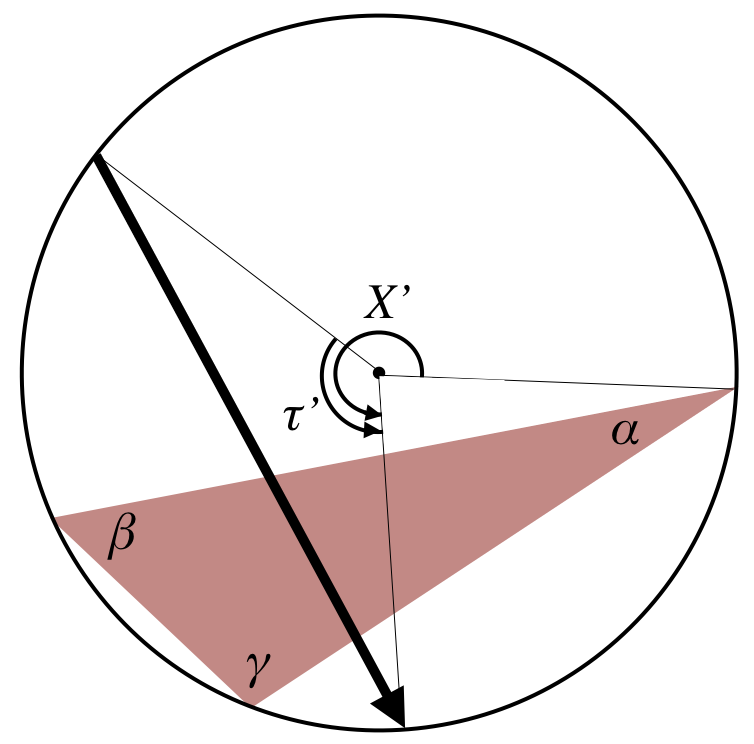}
\caption{This is a continuation of Figure \ref{xtau}. (a) The trajectory passes from the original (blue) triangle into the next (red) triangle. (b) We fold the new triangle onto the original triangle. (c) Then we reflect the new triangle as necessary so that $\alpha$, $\beta$ and $\gamma$ are in the counter-clockwise order. We see that $\tau'=\tau$, and that $X'$ is different from $X$.
\label{singlehit}}
\end{center}
\end{figure}

\subsection{From the circle to an interval exchange transformation}\label{sec:iet}

By keeping track of the location of an identified vertex of positively-oriented triangles rotating around the circumscribing circle, the dynamics of a trajectory on a triangle tiling can be described by an \emph{interval exchange transformation} (IET). 

First, we will describe how to construct an IET that contains all the information of a given trajectory, using the circumscribing circle from $\S \ref{sec:circle_setup}$.

For a given trajectory described by $(X,\tau)$ on a triangle tiling with angles $\alpha,\beta,\gamma$, there is an associated IET:
\begin{itemize}
\item The \emph{lengths of intervals} are determined by $\alpha,\beta,\gamma$ and $\tau$.
\item The \emph{shift transformations} are determined by $\alpha,\beta,\gamma$ and $\tau$.
\item The \emph{starting point} on the IET is determined by $X$.
\end{itemize}

\begin{theorem}[Tiling Billiards IET]\label{billiardsiet}
Given a triangle tiling with angles $\alpha$, $\beta$, $\gamma$, and a trajectory with associated parameters $(\tau,X)$, passing to the next triangle transforms $X$ according to the following 3-IET:

The interval $(0,2\gamma)$ maps to $(\tau,\tau+2\gamma)$, with the opposite orientation. The interval $(2\gamma,2\gamma+2\alpha)$ maps to $(\tau+2\gamma,\tau+2\gamma+2\alpha)$, with the opposite orientation. The interval $(2\gamma+2\alpha,2\pi)$ maps to $(\tau+2\gamma+2\alpha,\tau+2\pi)$, with with the opposite orientation. All of these transformations are taken modulo $2\pi$.
\end{theorem}

See Figure \ref{3ietfigure} for a picture of one such IET.

\begin{figure}[!h]
\begin{center}
\includegraphics[width=0.7\textwidth]{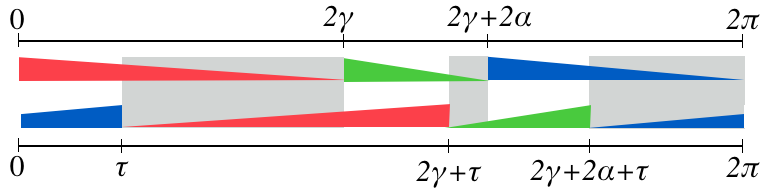}
\caption{The Tiling Billiards IET specified by a choice of $\alpha$, $\beta$, and $\tau$. The three subintervals are each flipped (orientation is indicated by the triangular shape of the intervals), and the entire interval is shifted to the right by $\tau$ modulo $2\pi$. The shaded regions indicate where the IET would say that the same side is hit twice in a row, which is impossible. In fact, these are not in the domain of the trajectory system, because they correspond to chords that are disjoint from the triangle. \label{3ietfigure}}
\end{center}
\end{figure}

\begin{proof}
First we will determine which side a trajectory crosses, depending on $X$. Starting at the vertex with angle $\alpha$ and moving counterclockwise around the triangle, the order of the sides is $C$, then $A$, and then $B$. From $0$ to $2\gamma$, this angle keeps the trajectory within side $C$. From $2\gamma$ to $2\gamma+2\alpha$, it is passing through side $A$. And from $2\gamma+2\alpha$ to $2\pi$, side $B$ is hit.

For the change in $X$, we first note that when the next triangle $T_2$ is folded onto the current triangle $T_1$, since the triangles share a perpendicular bisector and are folded along a corresponding edge, they are reflections of one another across that bisector. So, the back end of the oriented chord in $T_2$ is a reflection across this bisector from where the front end of the chord in $T_1$ is.
In this way, flipping the intervals sends the original front end of a chord to the back end of the chord in the next triangle. Adding the distance $\tau$ between the back end and the front end completes the relationship between the front ends of the segments in each triangle.
\end{proof}

\begin{remark} \label{rem:inequalities}
Explicitly, for a trajectory with parameters $(X,\tau)$, passing to the next triangle gives new parameters $(X',\tau')$, where $\tau'=\tau$, and

$\begin{displaystyle}
X' = 
\begin{cases}
\tau+2\gamma-X, & \text{if the side crossed is } C \ (0<X<2\gamma); \\
\tau-2\beta+2\gamma-X, & \text{if the side crossed is } A \ (2\gamma<X<2\gamma+2\alpha); \\
\tau-2\beta-X, & \text{if the side crossed is } B \ (2\gamma+2\alpha<X<2\pi). \\
\end{cases}
\end{displaystyle}$
\end{remark}

\subsection{From IETs to a polygon exchange transformation}\label{sec:pet}

The $3$-interval exchange just described is orientation reversing on each interval. For our purposes, it is easier to work with the orientation-preserving \emph{square} of the IET, which has between 3 and 6 intervals.

It turns out that it is natural to use a polygon exchange transformation (PET) for each triangle tiling, consisting of a stack of the squared, orientation-preserving IETs corresponding to all possible trajectories on the tiling. Each triangle tiling has an associated PET, which is a square ($2\pi\times 2\pi$) region cut up into triangles and parallelograms. Each point in the PET corresponds to a pair $(X,\tau)$, a particular starting point on a particular trajectory on the tiling. The exchange transformations are horizontal translations of the parallelogram regions. See Figure \ref{PETcolor} for an example of one such PET.

\begin{itemize}
\item The chord angle $\tau$ is measured along the vertical axis. For a given trajectory, $\tau$ is fixed (Lemma \ref{taufixed}), so a horizontal slice of the PET corresponds to a single trajectory's Tiling Billiards IET (Theorem \ref{billiardsiet}).
\item The triangle position $X$ is measured along the horizontal axis.
\item The shapes of the triangles and parallelograms that make up the pieces of the PET are determined by the angles $\alpha,\beta,\gamma$ of the triangle tiling.
\item The triangular regions at the top and bottom of the PET correspond to disallowed trajectories; they are not physically possible.
\item The shift transformations of the parallelograms are all horizontal, because the PET is constructed from stacks of horizontal IETs with shift transformations.
\end{itemize}

\begin{lemma}\label{redzone}
The following tuples $(X,\tau)$ are impossible:
\begin{enumerate}
\item $0<X < 2\gamma, \tau<X$
\item $2\gamma<X<2\gamma+2\alpha, \tau < X-2\gamma$
\item $2\gamma+2\alpha < X < 2\pi, \tau < X-(2\gamma+2\alpha)$
\item $0<X < 2\gamma, X < \tau - (2\alpha+2\beta)$
\item $2\gamma<X<2\gamma+2\alpha, X < \tau - 2\beta$
\item $2\gamma+2\alpha < X < 2\pi, X < \tau$
\end{enumerate}
\end{lemma}

\begin{proof}
Each of these follows from drawing a picture like Figure \ref{xtau}, and noticing that the chord does not intersect the triangle. Recall from basic geometry that the central angles of the arcs subtended by inscribed angles $\alpha, \beta$ and $\gamma$ are $2\alpha, 2\beta$ and $2\gamma$, respectively.

Cases (1), (2) and (3) occur when the (oriented)  chord of the trajectory lies outside the triangle, keeping the triangle on its left, in the arc spanned by sides $C$, $A$ or $B$, respectively.
Cases (4), (5) and (6) are the same with the triangle on the right.
\end{proof}

\begin{theorem}[Tiling Billiards PET]\label{thm:pet}
Take a triangle tiling with angles $\alpha$, $\beta$, and $\gamma$, and a trajectory with associated parameters $(X,\tau)$. When the trajectory reaches the next triangle of the same orientation, $(X,\tau)$ transforms according to the following PET:

\begin{align*}
  X'' &=
  \begin{cases}
    X+2\alpha, & \text{if\ } 0<X<2\gamma \text{\ and\ } X+2\alpha<\tau<X+2\alpha+2\beta \ (\text{side\ }C \text{,\ then\ } B) \\
    X-2\beta, & \text{if\ } 0<X<2\gamma \text{\ and\ } X<\tau<X+2\alpha \ (\text{side\ }C \text{,\ then\ } A) \\
    X+2\beta, & \text{if\ } 2\gamma<X<2\gamma+2\alpha \text{\ and\ } X-2\gamma+2\beta<\tau<X+2\beta \ (\text{side\ }A \text{,\ then\ } C) \\
    X-2\gamma, & \text{if\ } 2\gamma<X<2\gamma+2\alpha \text{\ and\ } X-2\gamma<\tau<X-2\gamma+2\beta \ (\text{side\ }A \text{,\ then\ } B) \\
    X+2\gamma, & \text{if\ } 2\gamma+2\alpha<X<2\pi \text{\ and\ } X-2\alpha<\tau<X \ (\text{side\ }B \text{,\ then\ } A) \\
    X-2\alpha, & \text{if\ } 2\gamma+2\alpha<X<2\pi \text{\ and\ } X-2\alpha-2\gamma<\tau<X-2\alpha \ (\text{side\ }B \text{,\ then\ } C)
  \end{cases}\\
\tau''&=\tau
\end{align*}

By Lemma \ref{redzone}, values of $(X,\tau)$ outside of these regions do not correspond to tiling billiards trajectories, so we take them as fixed points of the PET.
\end{theorem}

The polygon exchange transformation is shown in Figure \ref{PETcolor}.

\begin{figure}[!h]
\begin{center}
\includegraphics[width=0.45\textwidth]{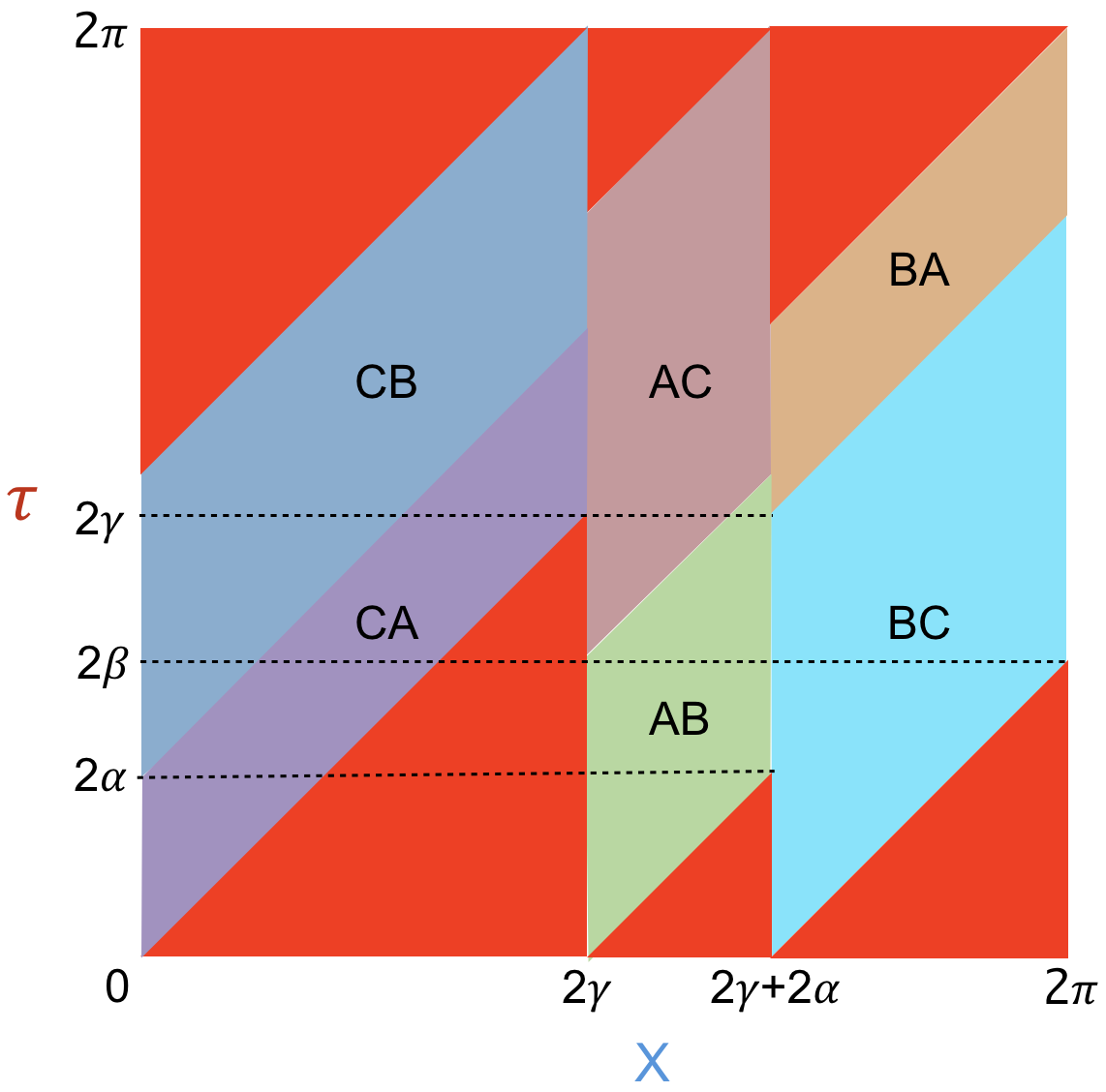} \hspace{0.03\textwidth}
\includegraphics[width=0.45\textwidth]{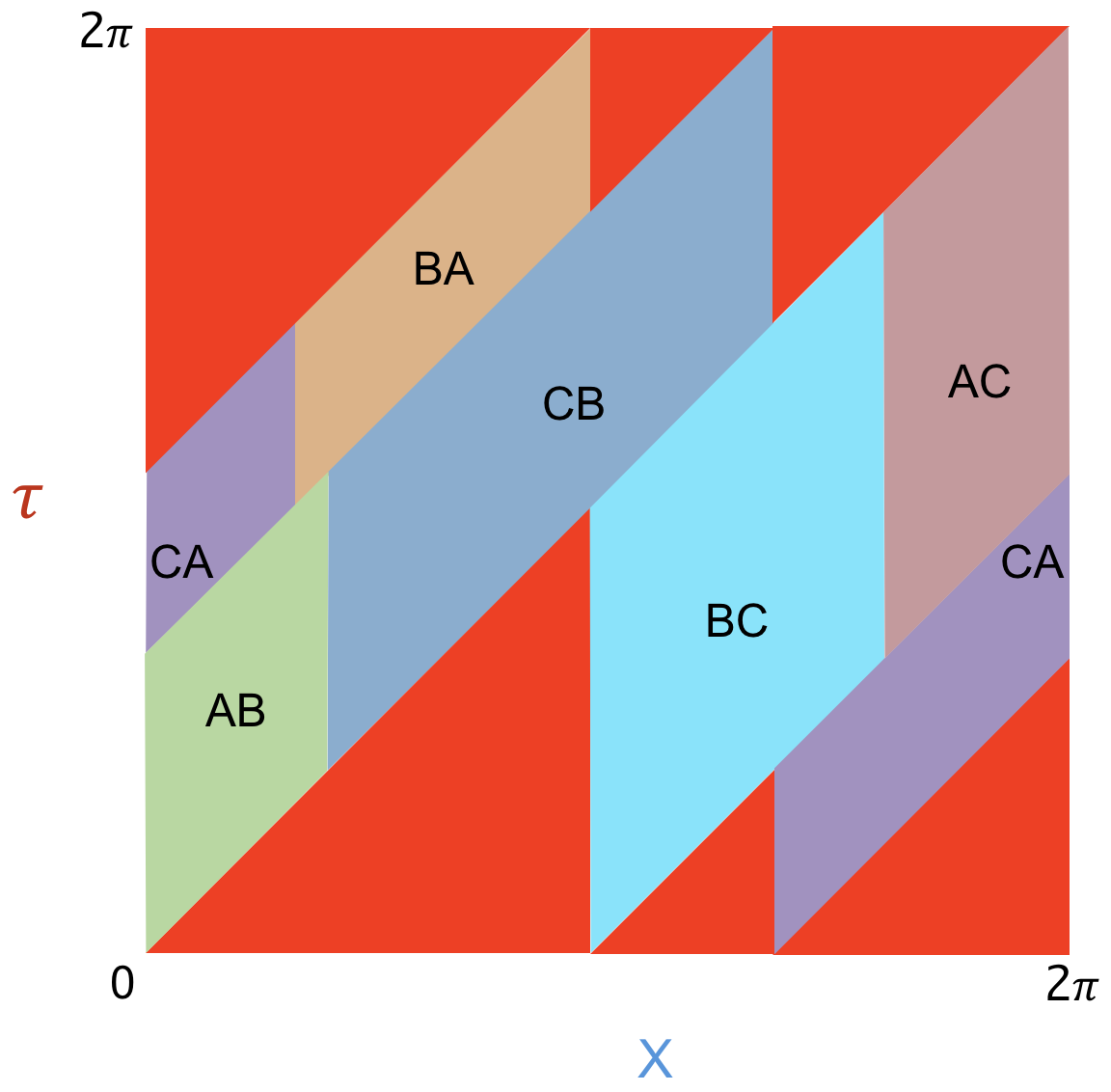}
\caption{The division of the $X,\tau$ plane into the polygons of the PET. The image on the left shows where the pieces start, and the image on the right shows where they go after one transformation. The red triangular regions represent disallowed trajectories, which we take to be fixed points. The label gives the next two edges that a trajectory corresponding to a point in that region will hit. \label{PETcolor}}
\end{center}
\end{figure}

\begin{proof}
This follows directly from the Tiling Billiards IET Theorem \ref{billiardsiet}. As an example, suppose that $0<X<2\gamma$ and $X+2\alpha<\tau<X+2\alpha+2\beta$. Then the theorem tells us that the first side hit is $C$. Applying the formula from Remark \ref{rem:inequalities}, we get \mbox{$X'=\tau+2\gamma-X$}. Combining this with the second inequality, we find $$(\tau+2\gamma-X')+2\alpha<\tau<(\tau+2\gamma-X')+2\alpha+2\beta.$$ This simplifies to \mbox{$2\gamma+2\alpha<X'<2\pi$}, which implies that the second side hit is $B$.

Since the sides crossed are $C$ and then $B$, we have 
\begin{align*}
X' &= \tau + 2\gamma -X; \\
X'' &= \tau-2\beta-(\tau+2\gamma-X)\\
&= X+2\alpha,
\end{align*}
with angles measured modulo $2\pi$.
\end{proof}

\begin{example} The PET corresponding to the equilateral triangle tiling is shown on the left side of in Figure \ref{pet-equil}a. A value of $\tau=5\pi/3$ slices the PET near the top, giving a $3$-interval IET corresponding to the blue, pink and brown regions, which are cyclically permuted in the exchange (top right in Figure \ref{pet-equil}). The red regions are fixed points. Every non-fixed point in this IET crosses the sequence $\overline{CBACBA}$ of edges, and corresponds to a trajectory circling a single vertex. Similarly, a value of $\tau = \pi/3$ slices the PET near the bottom, permuting the purple, green and cyan regions (bottom right in Figure \ref{pet-equil}). Non-fixed points in this IET correspond to the same trajectories as above, traversed in the opposite direction $\overline{ABCABC}$.

A slice through the middle of the IET, with $2\pi/3 < \tau < 4\pi/3$, yields a $6$-interval IET that is reducible into the two $3$-interval IETs described above (middle right in Figure \ref{pet-equil}b).
%
\end{example}

\begin{figure}[!h]
\begin{center}
\includegraphics[width=0.85\textwidth]{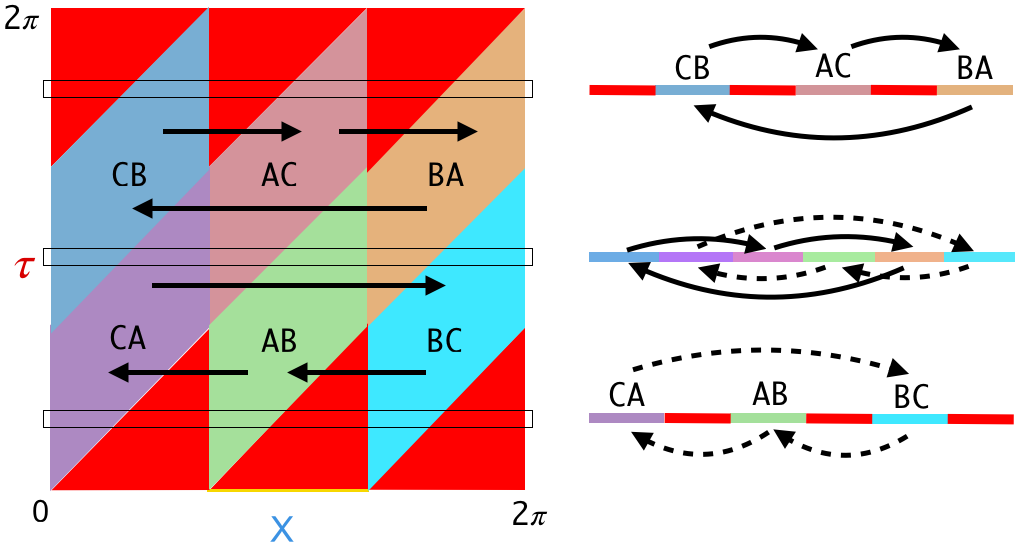}
\caption{(a) The PET for the equilateral triangle tiling. (b) Horizontal slices of the PET, giving the associated IETs. The top slice at $\tau=5\pi/3$ corresponds to a counter-clockwise trajectory around a single vertex. The bottom slice at $\tau=\pi/3$ corresponds to a clockwise trajectory around a single vertex. The middle slice is at $\tau=\pi$; in this case the direction of travel around the vertex depends on the starting direction $X$ of the trajectory. \label{pet-equil}}
\end{center}
\end{figure}

\begin{remark}
Sliding up and down the PET to choose different IETs (without hitting singularities) is the same as performing a rel deformation of the IET, and of the translation surface associated to the suspension of the IET; see Figure \ref{fig:suspension}.
\end{remark}

\begin{figure}[!h]
\begin{center}
\includegraphics[width=0.9\textwidth]{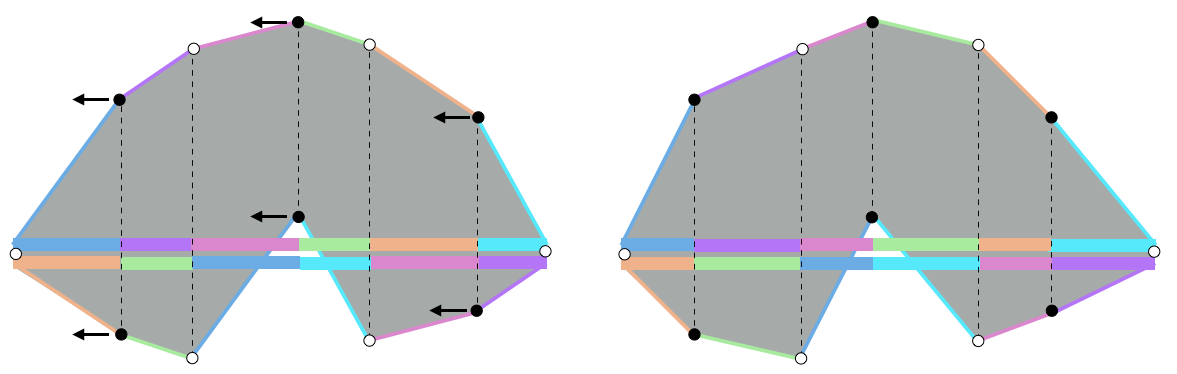}
\caption{Two IETs corresponding to slices of the PET for the equilateral triangle tiling, and their suspensions to translation surfaces, for (left) \mbox{$\pi<\tau<4\pi/3$} and (right) \mbox{$2\pi/3<\tau<\pi$}. The two IETs, and the two corresponding surfaces, are rel deformations of each other. \label{fig:suspension}}
\end{center}
\end{figure}

\begin{remark} \label{periodic-iet}
Periodic behavior of the IET corresponds to both periodic and drift-periodic trajectories in the tiling billiards system. This is because the IET records the position and direction of a trajectory in a triangle, but does not record the triangle's location in the tiling.
\end{remark}

\section{Results about tiling billiards from the IET and PET}

We now use the interval exchange transformation and polygon exchange transformation mechanisms developed in the previous section, to prove things about behavior of trajectories on the tilings.

\subsection{Results from the IET}

\begin{definition}\label{def:stable}
We call a trajectory on a triangle tiling \emph{stable} if there are neighborhoods of the given values for $\alpha,\beta,\gamma,\tau,$ and $X$ for which the sequence of edges crossed is unchanged.
\end{definition}

\begin{theorem}
All periodic trajectories are stable under small perturbations of the triangle angles $\alpha,\beta,\gamma$ and of the trajectory position parameters $X,\tau$.
\end{theorem}

\begin{proof}
By Theorem \ref{thm:sametriangle}, if a trajectory passes through the same triangle twice, it is in the same position both times. Thus for a periodic trajectory with period $N$, it is sufficient to show that the first $N$ edges hit are unchanged by sufficiently small perturbations in $\alpha,\beta,\gamma,X,$ and $\tau$.
As trajectories that hit vertices are not allowed, every value of $X$ in the finite sequence of side crossings is some positive distance away from the endpoints of the IET intervals from Theorem \ref{billiardsiet}. Let $\ve$ be the minimum such distance.

We claim that changing the triangle angles $\alpha,\beta,\gamma$ such that no angle changes by more than $\frac{\ve}{2N}$ does not change the sequence of edges crossed. Furthermore, changing $X$ or $\tau$ by less than $\ve/{2}$ does not change this sequence. That is, at an arbitrary iteration $k$ of the associated IET with $k\le N$, $X$ is in the same subinterval regardless of either such perturbation.

Below, we repeatedly use the Tiling Billiards IET (Theorem \ref{billiardsiet}), and the explicit equations given in Remark \ref{rem:inequalities} and Theorem \ref{thm:pet}, to demonstrate that for all $k\le N$, $|X^k-X^k_{\text{perturb}}|<\ve$.

If $k$ is even, we can get to $X^k$ from $X$ by $k/2$ applications of the squared transformation. Therefore, the Tiling Billiards IET Theorem \ref{billiardsiet} tells us that $X^k=X+p\alpha+q\beta+r\gamma$, for integers $p$, $q$, and $r$ with $|p|+|q|+|r|\le k$.

If $k$ is odd, we find $X^k$ from $X^{k-1}$ through one application of the equations in Remark \ref{rem:inequalities}. So, we find that $X^k=\tau-X+p\alpha+q\beta+r\gamma$, where $|p|+|q|+|r|\le k+3$.

If $k$ is even, we have $$X^k=X+2(\alpha n_{\alpha,k}+\beta n_{\beta,k}+\gamma n_{\gamma,k}),$$ where $$\sum_i{n_{i,k}}\le \frac{k}{2}<N \text{ for } i=\alpha,\beta,\gamma.$$

If $k$ is odd, we have $$X^k=\tau-X-2(\alpha n'_{\alpha,k}+\beta n'_{\beta,k}+\gamma n'_{\gamma,k}),$$ where $$\sum_i{n'_{i,k}}\le \sum_i{n_{i,k-1}}+2 \le \frac{k+1}{2}+1\le\frac{N}{2}+1<N.$$

For a perturbation of the triangle angles, let $\theta$ be the largest change in any of the angles. For $k$ even, we have $$|X^k-X^k_{\text{perturb}}|\le2\theta\sum_i{n_{i,k}}<2\theta N\le \ve.$$ For $k$ odd, $$|X^k-X^k_{\text{perturb}}|\le2\theta\sum_i{n'_{i,k}}<2\theta N\le \ve.$$

Changing the angles does affect the locations of the intervals of the IET, but changes of this size are negligible.

For perturbations of $X$ or $\tau$, the implication that $$|X-X_\text{perturb}|<\frac{\ve}{2} \text{ and } |\tau-\tau_{\text{perturb}}|<\frac{\ve}{2}$$ imply $$|X^k-X^k_{\text{perturb}}|<\ve$$ is straightforward.
\end{proof}

\begin{proposition}\label{rational-periodic}
All trajectories on rational triangle tilings are periodic or drift-periodic.
\end{proposition}

In this context, \emph{rational} means that every angle is a rational multiple of $\pi$.

\begin{proof}
The 6-IET associated to a trajectory on a rational triangle tiling has intervals of rational lengths. Thus, there exists some rational $\epsilon$ that divides the length of each interval of the IET. That is, for any given point on the interval, there is a finite number of places (at most $\pi / \epsilon$) it can land after repeated iterations of the IET. The image of the point under the IET, after some number of iterations, must return to the original point, and therefore the associated trajectory is either periodic or drift periodic (See Remark \ref{periodic-iet}).
\end{proof}

\begin{theorem}\label{rational-stable}
All 
trajectories on rational triangle tilings are stable under small perturbations of the trajectory position.
\end{theorem}

\begin{proof}
Since the angles of the triangle are rational, the associated IET to any trajectory on the tiling has intervals with rational lengths. There exists some rational $\lambda$ that divides the length of each interval. Split the entire interval into subintervals of length $\lambda$. For any point $X$ on the IET (corresponding to a particular trajectory), let $\epsilon$ be the distance from $X$ to nearest endpoint of a subinterval. Then, for an $\epsilon$ neighborhood of $X$, the IET construction guarantees the same behavior of the trajectory up to combinatorics.
\end{proof}

\subsection{Results from the PET}

\begin{proposition}\label{prop:onp}
On tilings by obtuse triangles, trajectories with $\tau$ between $2\alpha+2\beta$ and $2\gamma$ are not periodic.
\end{proposition}

\begin{proof}
Figure \ref{escape-obtuse} shows a PET diagram for a typical obtuse triangle. For $2\alpha+2\beta<\tau<2\gamma$, the IET reduces to two 2-interval IET's, each with interval lengths $2\alpha$ and $2\beta$. In particular, there are no $\gamma$-moves or $-\gamma$-moves, and a trajectory has a nonzero number of either $\alpha$-moves or $-\alpha$-moves (but not both) in a given orbit. Thus, $n_\alpha$ is never equal to $n_\gamma=0$, so by Lemma \ref{nx}, the trajectory can never close up.
\end{proof}

\begin{definition}
Consider the square torus, with horizontal edges labeled with one symbol and vertical edges labeled with another. A \emph{Sturmian sequence} is a bi-infinite sequence of the two symbols, which is a cutting sequence on the square torus corresponding to a trajectory with an irrational slope.
\end{definition}

We have shown that any trajectory with $2\alpha+2\beta<\tau<2\gamma$ is represented by a rotation of $2\alpha$ on a circle of circumference $2\alpha+2\beta$. For irrational $\alpha/\beta$, this characterizes a positive-measure class of aperiodic escaping trajectories. In particular, these trajectories go through edge $C$ in every triangle they pass, and the sequence of $A$s and $B$s passed in between is a Sturmian sequence.

As a consequence of our later Lemma \ref{gamma-obtuse} parts (2) and (4), for obtuse or right triangles, there are no aperiodic trajectories with $\tau$ outside of $(2\alpha+2\beta,2\gamma)$. For acute triangles, there are other positive-measure regions of aperiodic trajectories.

\begin{figure}[!h]
\begin{center}
\includegraphics[width=\textwidth]{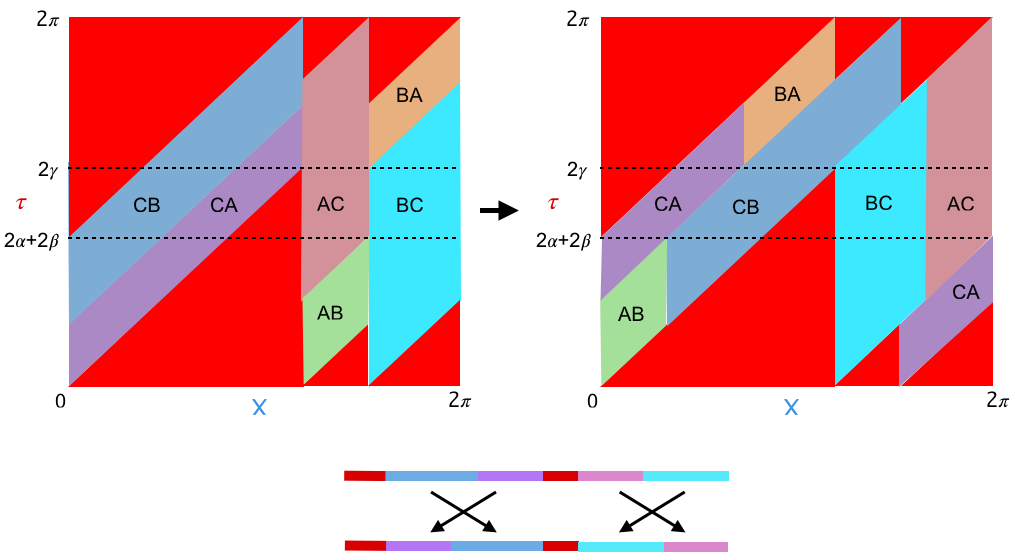}

\caption{This is a PET diagram for an obtuse triangle. Below is a slice taken from the central region, where the IET is reduced to two 2-IETs. \label{escape-obtuse}}
\end{center}
\end{figure}

\begin{definition}
A \emph{comoving region} in a PET is a maximal connected set of points that do the same thing under every iteration of the PET.
\end{definition}

\begin{proposition}\label{prop:bound}
If $\alpha,\beta$ and $\gamma$ are rational, with greatest common divisor $\epsilon$, then
\begin{enumerate}
\item The period of a periodic trajectory is bounded by $2\pi/\epsilon$, and
\item The number of different trajectories, up to combinatorics, is bounded by $\pi^2/(2\epsilon^2)$.
\end{enumerate}
\end{proposition}

\begin{proof}
\begin{enumerate}
\item By Theorem \ref{thm:pet}, the PET shifts points horizontally by $\pm 2\alpha, \pm 2\beta$, or $\pm 2\gamma$. Linear combinations of these quantities are always multiples of $2\epsilon$. As soon as a given point comes up for the second time in the orbit of a point, the orbit is periodic. On an interval of width $2\pi$, the greatest number of different points reachable by shifts that are multiples of $2\epsilon$ is $2\pi/2\epsilon = \pi/\epsilon$. Each point in the orbit corresponds to crossing $2$ edges, so we multiply by $2$.

\item When we cut the PET space $[0,2\pi]\times[0,2\pi]$ into comoving regions, there is a diagonal line at every $2\epsilon$ (measured in the horizontal direction), and a vertical line at every $2\epsilon$, so there are at most $\pi^2/\epsilon^2$ comoving regions in the $2\pi$ by $2\pi$ torus. The period of any trajectory (periodic or drift periodic) is at least 4, so each trajectory includes at least $2$ comoving regions and the number of distinct trajectories, up to combinatorics, is bounded by $\pi^2/(2\epsilon^2)$.
\end{enumerate}
\end{proof}

\begin{example}
The lowest possible period of a periodic trajectory is $6$, which can be achieved around any vertex of every triangle tiling, and corresponds an IET slice near the top or bottom of the PET, i.e. with $\tau$ close to $0$ or $2\pi$.
For the equilateral triangle tiling, Proposition \ref{prop:bound} gives a maximum period of $2\pi/\epsilon=6$, so every trajectory has period $6$.
\end{example}

\begin{example}
Figure \ref{components} shows the maximal comoving regions for the isosceles right triangle PET. The top and bottom (purple and blue) regions correspond to periodic trajectories of period $6$, the top one counter-clockwise and the bottom one clockwise. The middle (green and orange) regions correspond to drift-periodic trajectories of period $4$, whose sequences are offset by one letter. We conjecture (Conjecture \ref{central_tau}) that all escaping trajectories arise in this way, in comoving regions that intersect the line $\tau=\pi$:
\end{example}

\begin{figure}[!h]
\begin{center}
\includegraphics[width=7.5cm]{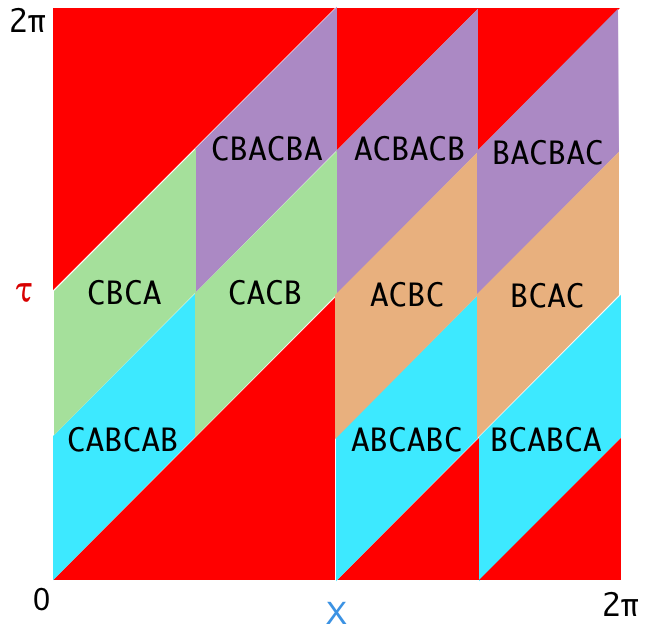}
\caption{Maximal comoving regions for the PET associated to the isosceles right triangle tiling, labeled with one period of the corresponding edge crossing sequence. Each iteration of the PET corresponds to hitting two sides. The corresponding periodic and drift-periodic trajectories are in Figure \ref{fig:per-dp}.  \label{components}}
\end{center}
\end{figure}

\begin{conjecture}\label{central_tau}
Every escaping trajectory is in a comoving region that intersects the line $\tau=\pi$. If a given triangle has a drift periodic trajectory, it is the only escaping trajectory up to combinatorics. In particular, this implies that drift periodic trajectories are read the same way forwards as backwards (they are palindromes -- see Figure \ref{palindrome}). If a triangle has an aperiodic trajectory, all other escaping trajectories on that triangle exhibit similar behavior. For instance, if one trajectory is a Sturmian sequence of subwords $v$ and $w$, the others are also a Sturmian sequences of $v$ and $w$, corresponding to the same irrational slope.
\end{conjecture}

\emph{Justfication for conjecture.} Suppose we have an orbit in our PET, representing an escaping trajectory, which is disjoint from $\tau=\pi$. Since the trajectory does not occur at $\tau=\pi$, intuition suggests that it relies on features specific to the outer regions of the PET. One thing we know is that when $\tau$ becomes small, the clockwise moves ($CA,AB,BC$) become larger portions of the IETs, while when $\tau$ gets closer to $2\pi$, the counterclockwise moves start winning out. If an escaping trajectory were to rely on one direction of rotation over the other, its total curvature would accumulate, implying either spiraling or crossing. Spiraling would have to go both ways, and it cannot continue indefinitely in the inward direction (Corollary \ref{nospiralnodense}). Crossing is similarly impossible (Theorem \ref{sametrajectory}).

In the aperiodic case, that there is only one such trajectory intersecting $\tau=\pi$ seems to relate to Theorem A of \cite{nogueiraflip}. Intuition and our experience with simulations in our computer program\footnote{Our computer program, which models tiling billiards on several different tilings, was written in Java by Pat Hooper and Alex St Laurent. It can be run in a browser here: \href{http://awstlaur.github.io/negsnel/}{http://awstlaur.github.io/negsnel/}} extend this to the drift periodic case.

\begin{figure}
\begin{center}
\includegraphics[height=150pt]{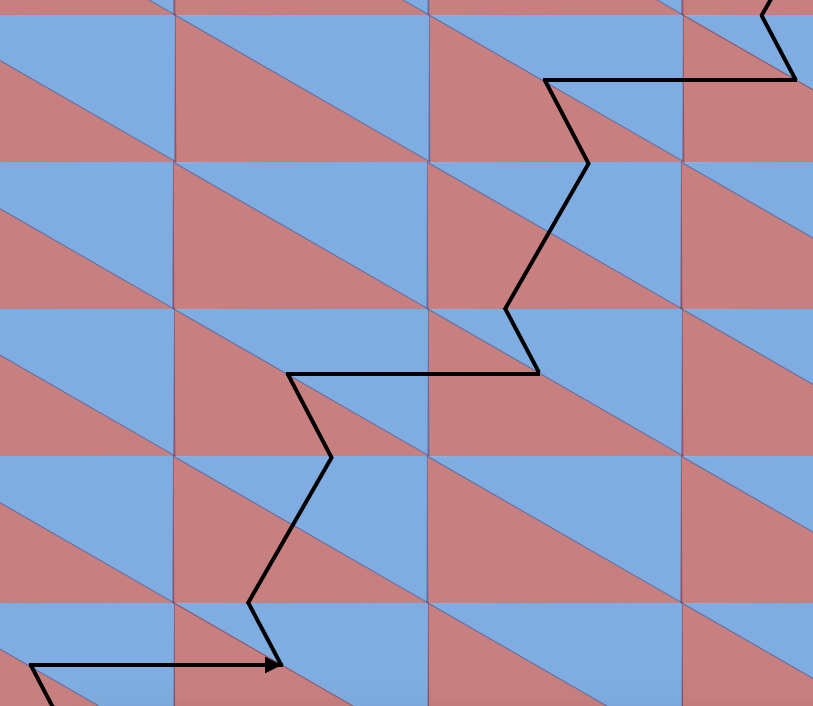} \qquad
\includegraphics[height=150pt]{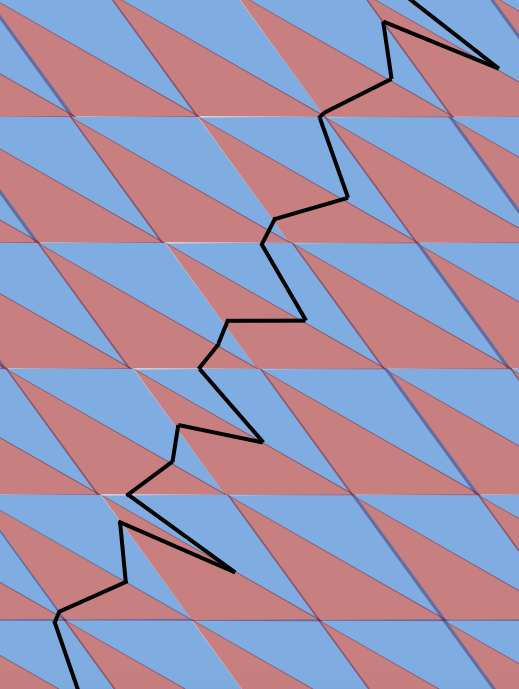}
\caption{(a) A drift-periodic trajectory on the $30^\circ-60^\circ-90^\circ$ triangle tiling, which cuts through edges $\overline{CBCBCA}$, no matter which direction we are traversing the trajectory. (b) A drift-periodic trajectory on the $126^\circ-30^\circ-24^\circ$ triangle tiling, which cuts through edges $\overline{BCBCACBCACBCACBC}$; this sequence is again the same in both directions. We believe (Conjecture \ref{central_tau}) that every drift-periodic trajectory exhibits this ``palindrome'' structure. \label{palindrome}}
\end{center}
\end{figure}

\subsection{The Arnoux-Yoccoz IET and its tiling billiards trajectory}

\begin{definition}\label{ayiet}
The \emph{Arnoux-Yoccoz interval exchange transformation} is defined on a circle of circumference $1$. Let $a$ be the real solution to the equation $x+x^2+x^3=1$. Define the following piecewise transformation on the circle (see \cite{arnoux}, \S 1):
\[f(x) = \begin{cases}
    x+\frac{a}{2}+\frac{1}{2} \ \mod 1, x \in \left[0, \frac{a}{2}\right) \\
    x-\frac{a}{2}+\frac{1}{2} \ \mod 1, x \in \left[\frac{a}{2},a\right) \\
    x+\frac{a^2}{2}+\frac{1}{2} \mod 1, x \in \left[a, a+\frac{a^2}{2}\right) \\
    x-\frac{a^2}{2}+\frac{1}{2} \mod 1, x \in \left[a+\frac{a^2}{2}, a+a^2\right)\\
    x+\frac{a^3}{2}+\frac{1}{2} \mod 1, x \in \left[a+a^2, a+a^2+\frac{a^3}{2}\right)\\
    x-\frac{a^3}{2}+\frac{1}{2} \mod 1, x \in \left[a+a^2+\frac{a^3}{2}, 1\right)\\
  \end{cases}\] \vspace{-1em}

Geometrically, swap the pairs of subintervals of equal length, and then rotate by a half turn.
\end{definition}


\begin{theorem}\label{thm:ayiet}
For $\alpha = \frac{\pi(1-a)}{2}$, $\beta = \frac{\pi(1-a^2)}{2}$, and $\tau = \pi$, the square of the tiling billiards IET (defined in Theorem \ref{billiardsiet}) is conjugate to the Arnoux-Yoccoz IET by a linear transformation.
\end{theorem}

\begin{proof}
Set $\alpha = \frac{\pi(1-a)}{2}$, $\beta = \frac{\pi(1- a^2)}{2}$, $\gamma = \frac{\pi(1- a^3)}{2}$, and $\tau = \pi$ in the tiling billiards system. The resulting IET, by Theorem \ref{billiardsiet}, is described by the following piecewise transformation on the interval $[0, 2\pi]$ with endpoints identified: \vspace{-1em}

\[g(y) = \begin{cases}
    y+2\alpha = y-\pi{a}+\pi \ \mod 2\pi, y \in [0, \pi a) \\
    y- 2\beta = y+\pi{a^2}+\pi \mod 2\pi, y \in [\pi{a}, \pi(a+a^2))\\
    y+2\beta = y-\pi{a^2}+\pi \mod 2\pi, y \in [\pi(a+a^2), \pi(a+2a^2)) \\
    y-2\gamma = y+\pi{a^3}+\pi \mod 2\pi, y \in [\pi(a+2a^2), \pi(a+2a^2+a^3))\\
    y+2\gamma = y-\pi{a^3}+\pi \mod 2\pi, y \in [\pi(a+2a^2+a^3), \pi(a+2a^2+2a^3)) \\
    y-2\alpha = y+\pi{a}+\pi \ \mod 2\pi, y \in [\pi(a+2a^2+2a^3), 2\pi)
  \end{cases}\] \vspace{-1em}

Shifting forward by $\pi a$ and scaling down to the interval $[0,1)$ recovers the Arnoux-Yoccoz IET of Definition \ref{ayiet}.
\end{proof}

\begin {proposition}
The trajectory corresponding to $\alpha = \frac{\pi(1-a)}{2}$, $\beta = \frac{\pi(1-a^2)}{2}$, and $\tau = \pi$ is escaping for every value of $X$.
\end{proposition}


\begin{proof}
Let $X \in [0, 2 \pi)$ correspond to an orientation of a starting triangle. $X$ has a dense (aperiodic) orbit under the Arnoux-Yoccoz IET \cite[\S 2.2]{arnoux}. That is, the trajectory never returns to a triangle in the same orientation as the starting triangle, so it is not periodic. Thus by Theorem \ref{sametrajectory} part (\ref{nonperiodicescapes}), the trajectory is escaping.
\end{proof}

%

Using $\alpha = \frac{\pi(1-a)}{2}$ and $\beta = \frac{\pi(1-a^2)}{2}$, and approximating $\tau = \pi$ in our computer program (see footnote), we found trajectories that approach the Rauzy fractal, the first 12 of which are shown in Appendix \ref{rauzyfrac}.

\begin{remark}\label{rem:Hooper-Weiss}
Hooper and Weiss \cite{pat} depict the algebraic dynamics of rel deformations of the Arnoux-Yoccoz IET, using a Cayley graph of the subgroup of $\R/\Z$ containing the three shift values, $(1-a)/2$, $(1-a^2)/2$, $(1-a^3)/2$ (after multiplication by $2\pi$, these are equal to our $2\alpha$, $2\beta$, and $2\gamma$). Indeed, the IET given just after \cite[Proposition 4.6]{pat} is identical to our tiling billiards squared IET, after scaling, rotation, and the change of variables $r\mapsto \pi-\tau$:

In our system, changes by these values correspond to movement between positively oriented triangles in the plane. So, this Cayley graph can be found on our tiled plane, with vertices in positively oriented triangles, and our trajectories encode the same algebraic dynamics as can be seen in \cite[Figure 2]{pat}, reproduced here as Figure \ref{pat-fractal}. In particular, their real rel parameter $r$, which parametrizes their rel-deformed Arnoux-Yoccoz IETs, corresponds to the value of $\pi-\tau$. They make a corresponding observation that as $r$ gets closer to $0$, the curves corresponding to orbits on the IET fill larger regions of the Cayley graph. See the end of \cite[\S 4]{pat}.
\end{remark}

For another perspective on the closed loops obtained by changing the relative periods of the Arnoux-Yoccoz examples, see \cite[\S 5 and its two figures]{curt}.

\begin{figure}[!h]
\begin{center}
\includegraphics[height=200pt]{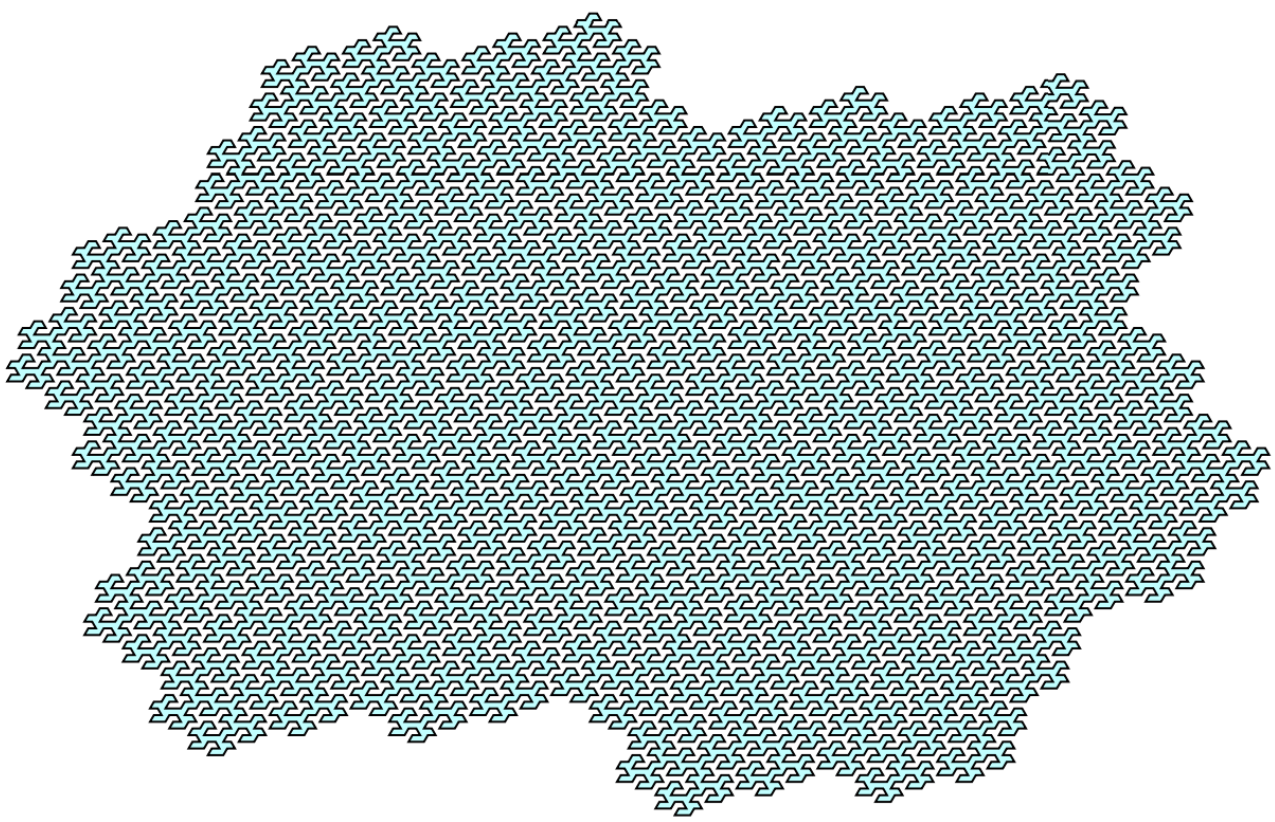}
\caption{An orbit from Hooper and Weiss's construction \cite{pat}; compare to our tiling billiards trajectories in Appendix \ref{rauzyfrac}. See Remark \ref{rem:Hooper-Weiss}. \label{pat-fractal}}
\end{center}
\end{figure}


\begin{theorem}
Consider the triangle tiling with $\alpha = \frac{\pi(1-a)}{2}$, and $\beta = \frac{\pi(1-a^2)}{2}$, in the folded position. Any chord with $\tau=\pi$ and $X\in\pi\Q(a)$ unfolds to finitely many trajectories in the tiling plane. Furthermore, some such chords unfold to a single trajectory that visits every triangle in the plane.
\end{theorem}

\begin{proof}
This follows from Theorem $2$ of Lowenstein, Poggiaspalla, and Vivaldi \cite{cubic}. Their work focuses on a scale-invariant first return map $\rho$ which is induced by the Arnoux-Yoccoz IET. For each point $P$ on their IET, they define a corresponding lattice $L_P$ which contains every point that differs from the starting point by a sum of translation distances (modulo 1). Theorem $2$ of \cite{cubic} states that for all points in $\Q(a)$, the corresponding lattice is composed of finitely many orbits. This immediately implies that the same is true of the Arnoux-Yoccoz IET. The lattice $L_P$ in the Arnoux-Yoccoz IET corresponds to the collection of $X$ values for positively oriented triangles with $\tau=\pi$ and $X$ equal to $P$ in a given triangle. Thus, positively oriented triangles are hit by finitely many trajectories in this unfolded trajectory. The result including negatively oriented triangles follows immediately.

\S4.6 of \cite{cubic} gives explicit examples of lattices that are composed of a single orbit.
\end{proof}


A trajectory with $\tau=\pi$ passes through the circumcenter of the tiling triangle. Approximating this as closely as possible in our computer program yields the longer and longer periodic trajectories shown in Appendix \ref{rauzyfrac}. In the limit where $\tau=\pi$, the trajectories escape, filling up a larger region of the plane.

\subsection{Trajectories corresponding to other points in the Rauzy gasket}

Recall that $a\approx 0.54$ is the real solution to the equation $x+x^2+x^3=1$. The Arnoux-Yoccoz point $(a,a^2,a^3)$ is the central point of the \emph{Rauzy gasket} (see \cite{gasket}), defined as follows.

\begin{definition}[Rauzy gasket]
Consider a point $(x_1,x_2,x_3)$ in $\mathbf{R}^3$ such that \mbox{$x_1+x_2+x_3=1$}. Apply the following algorithm:
\begin{enumerate}
\item If one of the entries, say $x_1$, is greater than the sum of the two smaller ones, $x_2+x_3$, subtract the sum of the two smaller ones from the largest, to get a new point \mbox{$(x_1-x_2-x_3,x_2,x_3)$} with positive entries.
\item Rescale so that the sum of the entries is $1$.
\end{enumerate}
The \emph{Rauzy gasket} is the set of points in $\mathbf{R}^3$ whose entries sum to $1$, and on which we can apply this algorithm infinitely many times, i.e. where step (1) is always possible.
\end{definition}

Since the Arnoux-Yoccoz point $(a,a^2,a^3)$ gives us a tiling where tiling billiards trajectories take the shape of the Rauzy fractal, we explored what the trajectories look like for other points in the Rauzy gasket. We define the corresponding tilings as follows:

\begin{definition}\label{def:rauzytriangle}
For tiling billiards, the triangle tiling associated to a point $(x_1, x_2, x_3)$ in the Rauzy gasket has angles
$\alpha = \pi(1-x_1)/2$,
$\beta = \pi(1-x_2)/2$,
$\gamma = \pi(1-x_3)/2$.
\end{definition}

The Arnoux-Yoccoz triangle (Theorem \ref{thm:ayiet}) is a special case of this formula, with $x_i = a^i$.
%
Trajectories on tilings corresponding to other points in the Rauzy gasket also look like fractals, of a different shape but with similar behavior to the pictures in Appendix \ref{rauzyfrac}, leading to the following conjectures:

\begin{conjecture}
Let $(x_1, x_2, x_3)$ be a point in the Rauzy gasket, and consider the triangle tiling with angles $\alpha,\beta,\gamma$ as in Definition \ref{def:rauzytriangle}. Then:
\begin{enumerate}
\item every trajectory through the circumcenter (i.e with $\tau=\pi$) is escaping, and \label{conj:taupi}
\item trajectories that pass increasingly close to the circumcenter (i.e. with $\tau\to\pi^-$) exhibit a fractal structure under rescaling.
\end{enumerate}
\end{conjecture}

After reading an early draft of this paper, Hubert and Paris-Romaskevich prove \eqref{conj:taupi} as the title result of their paper, and they also prove that this is the only way to get a nonlinearly escaping trajectory \cite{olga}.


\section{Periodic trajectories and the Tree Conjecture} \label{sec:tree}

Previous work on tiling billiards conjectured, based on computer evidence, that periodic trajectories on triangle tilings have period $4n+2$, $n \geq 1$:

\begin{conjecture}[$4n+2$ Conjecture, from Conjecture 4.12 in \cite{icerm}]\label{pet-4n-plus-2}
Every periodic trajectory on a triangle tiling has period $4n+2$ for some positive integer $n$.
\end{conjecture}

We give further justification:

\begin{remark}[Justfication for $4n+2$ conjecture]
Since almost every trajectory on a fully flipped IET is periodic, and almost all IETs have only robust minimal components, we know that almost every fully flipped IET consists of only periodic components. As per \cite{nonergodic}, this tells us that Rauzy induction ends after finitely many steps, and the transformation has a finite expansion. In the proof of Corollary 3.3 in \cite{nonergodic}, Nogueira states that in the irrational case, a finite expansion implies that periodicity comes from flipped fixed points. Without the use of a flipped fixed point, rational relations between the interval lengths are needed to create a periodic orbit. 
In our IETs, every interval is flipped, so in order to have a flipped fixed point, there has to have been an odd number of iterations. So the periodic component returns after twice an odd number, which is a number of the form $4n+2$.
\end{remark}

After reading an early draft of this paper, Hubert and Paris-Romaskevich proved this conjecture \cite[\S 6.4 and Theorem 7]{olga}.

We conjecture a stronger result, that a periodic trajectory never encloses a triangle. If we consider the tiling as vertices and edges of a planar graph, this means that the vertices and edges enclosed by a periodic trajectory form a \emph{tree} (Figure \ref{nicetree}):

\begin{conjecture}\label{conj:only-tree}[Tree Conjecture]
A periodic trajectory on a triangle tiling encloses a tree.
\end{conjecture}

\begin{figure}[!h]
\begin{center}
\includegraphics[width=0.5\textwidth]{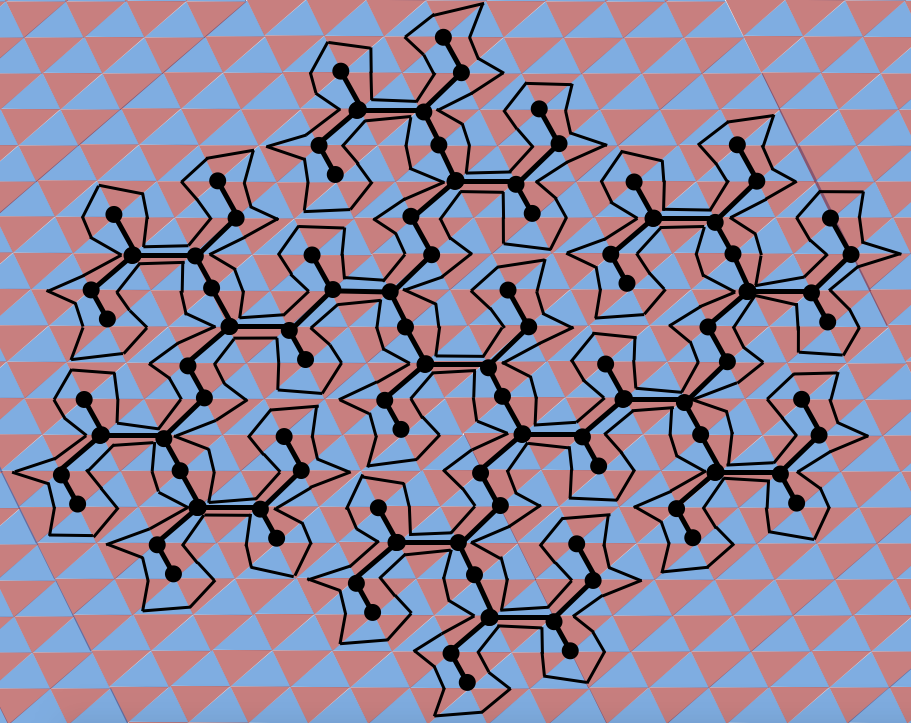}
\caption{A periodic path (thin) enclosing a tree (thick) of vertices and edges of the tiling. This is the eighth Rauzy fractal from Appendix \ref{rauzyfrac}. \label{nicetree}}
\end{center}
\end{figure}

In Theorem \ref{thm:tree}, we show that a periodic trajectory on an \emph{obtuse} triangle tiling encloses a tree (in fact, a path), proving the Tree Conjecture for the obtuse case.
In Proposition \ref{4n-plus-2}, we show that the $4n+2$ Conjecture follows from the Tree Conjecture, thus proving the $4n+2$ Conjecture for the obtuse case as well. In addition, we give meaning to the number $n$ in the expression $4n+2$, as the number of vertices enclosed by the periodic path.

\subsection{A periodic trajectory on an obtuse triangle tiling encloses a tree}
We thank Dylan Thurston for suggesting the reasoning in this section, that the region enclosed by a periodic trajectory is on the ``small side'' of the chord in the folded position. See Figure \ref{fig:hands}(c).

Recall from Definition \ref{def:moves} and Table \ref{table:moves} that a $\gamma$-move corresponds to edge crossings $BA$ and a $-\gamma$-move corresponds to edge crossings $AB$, each with corresponding regions in the PET (Figure \ref{PETcolor}).

\begin{lemma}\label{gamma-obtuse}
Given a trajectory on an obtuse triangle tiling with $\tau>2\gamma$ (symmetric results hold for $\tau<2\alpha+2\beta$):
\begin{enumerate}\item The trajectory has no $-\gamma$-moves.
\item If $\alpha/\beta$ is irrational, then the trajectory has a $\gamma$-move.
\item If the trajectory has a $\gamma$-move, then it has infinitely many $\gamma$-moves.
\item If the trajectory has a $\gamma$-move, then it is periodic with exactly one $\gamma$-move per period.
\end{enumerate}
\end{lemma}

\begin{proof}
\begin{enumerate}
\item This is straightforward from the PET and the definition of an obtuse triangle, as a $-\gamma$-move requires $\tau<2\alpha+2\beta$, and we have \mbox{$\tau>2\gamma>2\alpha+2\beta$}.

\item Let $\alpha/\beta$ be irrational. We show that a trajectory starting with $2\gamma<X<2\pi$ eventually reaches $0<X<2\gamma$, and that a trajectory with $0<X<2\gamma$ eventually finds $2\gamma+2\alpha<X<\tau+2\alpha$, corresponding to a $\gamma$-move. A trajectory with $2\gamma<X<2\pi$ can have $\alpha$-moves and $-\beta$-moves. This corresponds to a rotation of $2\alpha$ through the interval between $\tau-2\alpha-2\beta$ and $\tau$. By our assumption, this rotation is irrational, and the region between $2\gamma$ and $2\tau$ is eventually hit. A similar argument (for irrational rotation with $\beta$-moves and $-\alpha$-moves) shows that a trajectory with $2\gamma<X<2\pi$ will eventually hit the region between $2\gamma+2\alpha$ and $\tau+2\alpha$.

\item We claim that if a trajectory has a $\gamma$-move, then it will have another $\gamma$-move, and hence infinitely many. If $\alpha/\beta$ is irrational, the above argument suffices. Assume $\alpha/\beta$ is rational. From the PET, we find that after the $\gamma$-move, we have $2\gamma-2\beta<X<\tau-2\beta$. From here, rotating by $2\alpha$ cannot continue forever, since $X+2\beta$ would eventually be hit. Once we have $X<2\alpha+2\beta$, we know that the distance to our starting point is divisible by the GCD of $2\alpha$ and $2\beta$, so the same argument proves that the trajectory must return to the interval $2\gamma+2\alpha<X<\tau+2\alpha$, corresponding to another $\gamma$-move.

\item Without loss of generality, suppose the trajectory starts in a positively-oriented triangle. We will keep track of whether the trajectory is moving upwards or downwards (since no edge is vertical, crossing any edge corresponds to going either up or down). We will show that we only change vertical directions twice before retracing our path.
To switch from going up to going down, the trajectory must make a $\pm\gamma$-move (see Figure \ref{fig:gamma_moves}). Symmetrically, to switch from going down to going up, there must be a $\pm\gamma$-move when considering the upside-down triangles instead. Since there are no $-\gamma$-moves, every time we switch directions, we do so by going counter-clockwise. Suppose that more than one $\gamma$-move is made without the trajectory closing up. Then the trajectory must either cross itself (not allowed), or spiral indefinitely (not allowed), so this is impossible. Thus, it is periodic after one $\gamma$-move.
\end{enumerate}
\end{proof}

\begin{figure}[!h]
\begin{center}
\includegraphics[height=300pt]{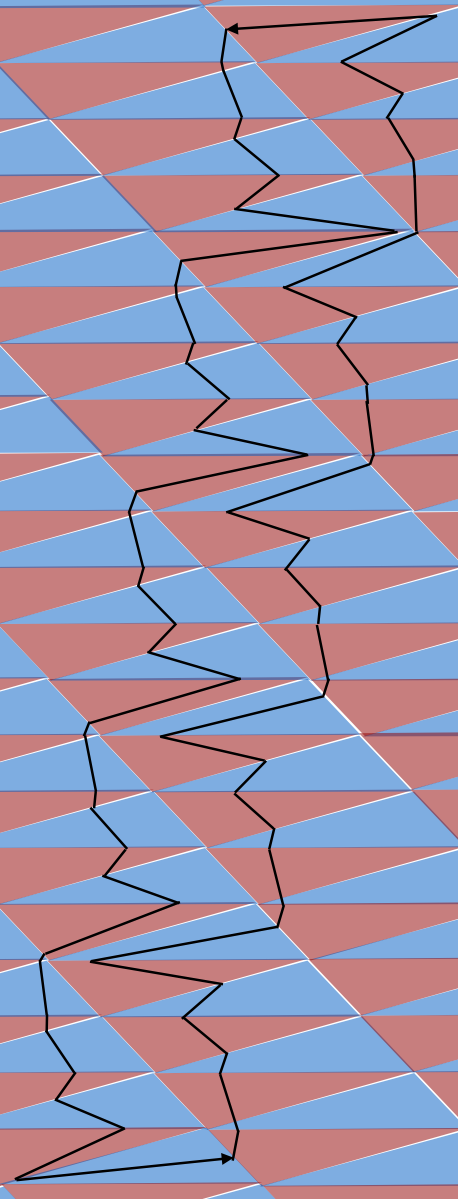} \hspace{1cm}
\includegraphics[height=300pt]{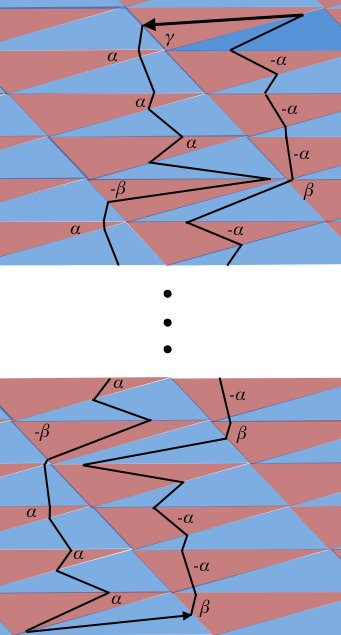}
\caption{(left) A closed trajectory on an obtuse triangle tiling has (right) at most one $\gamma$-move (here, at the top). On the right, moves are labeled. \label{fig:gamma_moves}}
\end{center}
\end{figure}

\begin{corollary}\label{cor:obtuse-per-dp}
For an obtuse triangle tiling, if $\tau>2\gamma$ or $\tau<2\alpha+2\beta$, then the trajectory is either periodic or drift-periodic.
\end{corollary}

\begin{proof}
If $\alpha/\beta$ is irrational, then combining (2) with (4) gives that the trajectory is periodic. If $\alpha/\beta$ is rational, and the trajectory is not periodic, then there are no $\pm\gamma$-moves, by part (1) and the contrapositive of part (4). In this case, the point corresponding to the trajectory follows a rational rotation in the PET, so it must be either periodic or drift-periodic. Since we assumed it is not periodic, it is drift-periodic.
\end{proof}

\begin{corollary}\label{cor:small_side}
For a periodic trajectory on an obtuse triangle tiling, the region enclosed by the trajectory lies, in the folded position, on the small side of the chord in the circumscribing circle.
\end{corollary}

\begin{proof}
By Lemma \ref{gamma-obtuse} (1), $\tau>2\gamma$ means that the trajectory is traveling counter-clockwise, and $\tau<2\alpha+2\beta$ means that the trajectory is traveling clockwise. For a counter-clockwise trajectory, the enclosed region is on the left. Since $\tau$ is large, the region on the left is the small side of the chord.
\end{proof}

\begin{theorem}\label{thm:tree}
A periodic trajectory on an obtuse triangle tiling encloses a path.
\end{theorem}

\begin{proof}
Consider a triangle $T_1$ crossed by the trajectory, for which the edge not intersected by the trajectory is in the region enclosed by the trajectory. By Corollary~\ref{cor:small_side}, the enclosed edge lies on the small side of the chord. Thus in the folded state, the chord of the trajectory crosses the triangle $T_2$ on the other side of this edge, since the triangle is too large to fit entirely on this small side.

Since $T_2$ intersects the chord when folded, it is either crossed by the original trajectory or by a different trajectory sharing the same chord in the folded state. Again by Corollary~\ref{cor:small_side}, the edge shared by these two triangles lies on the small side of the chord, so it must be in the interior of both trajectories. Trajectories sharing a chord cannot intersect, so these two segments must be parts of the same trajectory.

Therefore if a trajectory encloses an edge, the trajectory must cross the triangles on both sides of that edge, so no trajectory encloses a triangle; a trajectory encloses only vertices and edges of the tiling.
Since a trajectory can enclose at most one edge of each triangle, and the trajectory chord is too short to enclose a long edge, the tree enclosed is actually a path.
\end{proof}

\subsection{The tree conjecture implies the 4n+2 conjecture}

Previous work on triangle tilings conjectured (\cite{icerm}, Conjecture 4.12) that every periodic trajectory has a period of the form $4n+2,n\geq 1$. Our Tree Conjecture \ref{conj:only-tree} is stronger, that every periodic trajectory encloses a tree, and $n$ is the number of enclosed vertices. We show in Proposition \ref{4n-plus-2} that the Tree Conjecture implies the $4n+2$ conjecture.

\begin {lemma} \label{alltriangle}
Suppose a periodic trajectory on a triangle tiling encloses a tree. If $v$ is a vertex of that tree, the trajectory must enter all six of the triangles adjacent to $v$.
\end{lemma}

\begin{proof}
Let $t$ be a periodic trajectory on a triangle tiling enclosing a vertex $v$. Suppose for contradiction that there is a triangle adjacent to $v$ that $t$ does not enter. Since $t$ is closed and simple and completely encloses vertex $v$, yet also does not intersect one of the adjacent triangles, it must completely enclose the triangle, so $t$ is not a tree.
\end{proof}

\begin{proposition}\label{4n-plus-2}
If a periodic trajectory on a triangle tiling encloses a tree with $n$ vertices, then it has period $4n+2$.
\end{proposition}

\begin{proof}
Let $t$ be a periodic path on a triangle tiling enclosing $n$ vertices. Each vertex in the tiling is adjacent to six triangles. Label the vertices of the tree enclosed by $t$ as $v_1,...,v_n$ such that each vertex $v_j$ with $j \neq 1$ is adjacent to some vertex $v_i$ with $i<j$. We will now count the period of the path by attributing certain parts of the path to each labeled vertex (see Figure \ref{4n2figure}).
By Lemma \ref{alltriangle}, $t$ intersects all six triangles adjacent to $v_1$. We know $v_2$ is adjacent to $v_1$ and therefore two of the six triangles adjacent to $v_2$ are also adjacent to $v_1$. By Theorem \ref{thm:sametriangle}, $t$ does not cross itself and thus exactly two pieces of $t$ intersecting triangles adjacent to $v_2$ have already been counted by $v_1$. That is, $v_2$ provides an additional four pieces of $t$. $v_3$ is either adjacent to $v_1$ or $v_2$ ($v_3$ cannot be adjacent to both, because then the trajectory would  enclose a triangle). Thus, $v_3$ shares two adjacent triangles with an already-counted vertex and adds four new pieces of $t$ to our count. Similarly, each of the vertices $v_4, \dots ,v_n$ is adjacent to exactly one of the already-counted vertices, and thus is responsible for four new pieces of $t$.
Thus the period of $t$ is $6+4(n-1)=4n+2$.
\end{proof}

\begin{figure}[!h]
\begin{center}
\includegraphics[width=0.5\textwidth]{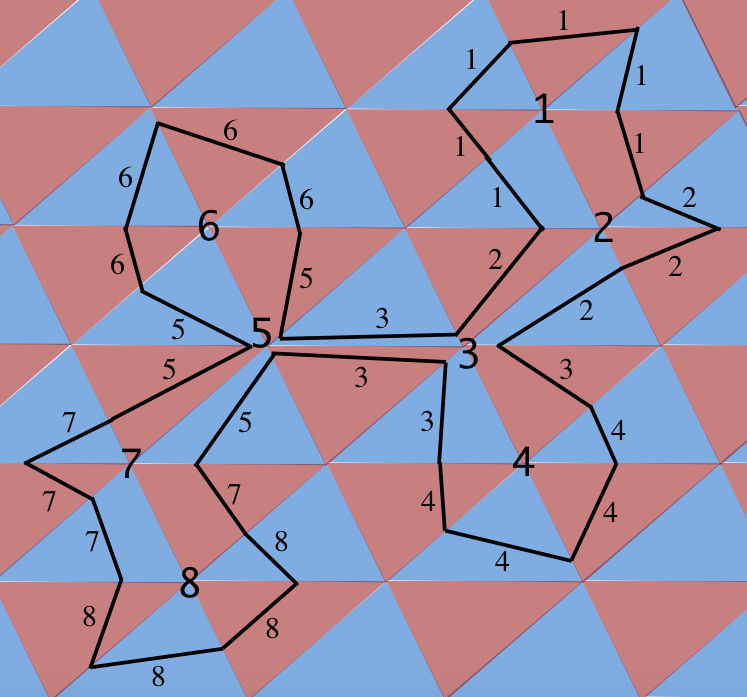}
\caption{The method of attributing pieces of a periodic trajectory (labeled with small text) to nearby vertices (labeled with large text) of the enclosed tree. There are six pieces attributed to the first vertex, and four new pieces for each additional vertex. The tree structure allows for systematic numbering. This trajectory is the fourth Rauzy fractal from Appendix \ref{rauzyfrac}. \label{4n2figure}}
\end{center}
\end{figure}

\begin{corollary}\label{cor:4n2obtuse}
A periodic trajectory on an obtuse triangle tiling enclosing $n$ vertices has period $4n+2$.
\end{corollary}

\vfill
\hrule

{\bf Paul Baird-Smith}, Department of Computer Science, University of Texas at Austin, 2317 Speedway, Austin, TX 78712, \textsf{ppb336@cs.utexas.edu}

{\bf Diana Davis}, Department of Mathematics and Statistics, Swarthmore College, 500 College Avenue, Swarthmore, PA 19081, \textsf{ddavis3@swarthmore.edu}

{\bf Elijah Fromm}, Department of Mathematics, Yale University, 10 Hillhouse Avenue, New Haven, CT 06511, \textsf{elijah.fromm@yale.edu}

{\bf Sumun Iyer}, Department of Mathematics, Cornell University, 310 Malott Hall, Ithaca, New York 14853, \textsf{ssi22@cornell.edu}

\newpage

\appendix

\section{Periodic trajectories approaching the Rauzy fractal}\label{rauzyfrac}

\vspace{-0.5em}

Trajectories on the triangle tiling associated to the Arnoux-Yoccox IET, of increasing length. Note the similarities down each column. The sets of triangles hit by each trajectory, and indeed the shape of the trajectories themselves, form a \emph{Tribonacci} sequence: pasting together three successive regions produces the next set of triangles in the sequence.

\begin{center}
\includegraphics[width=4cm]{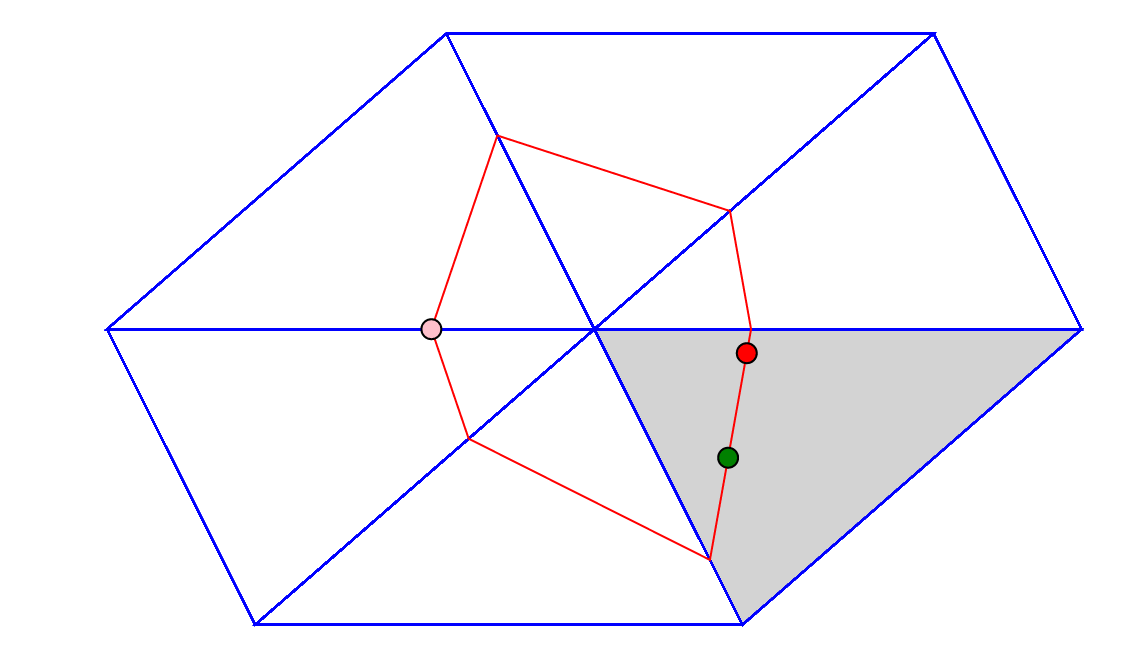}
\includegraphics[width=4cm]{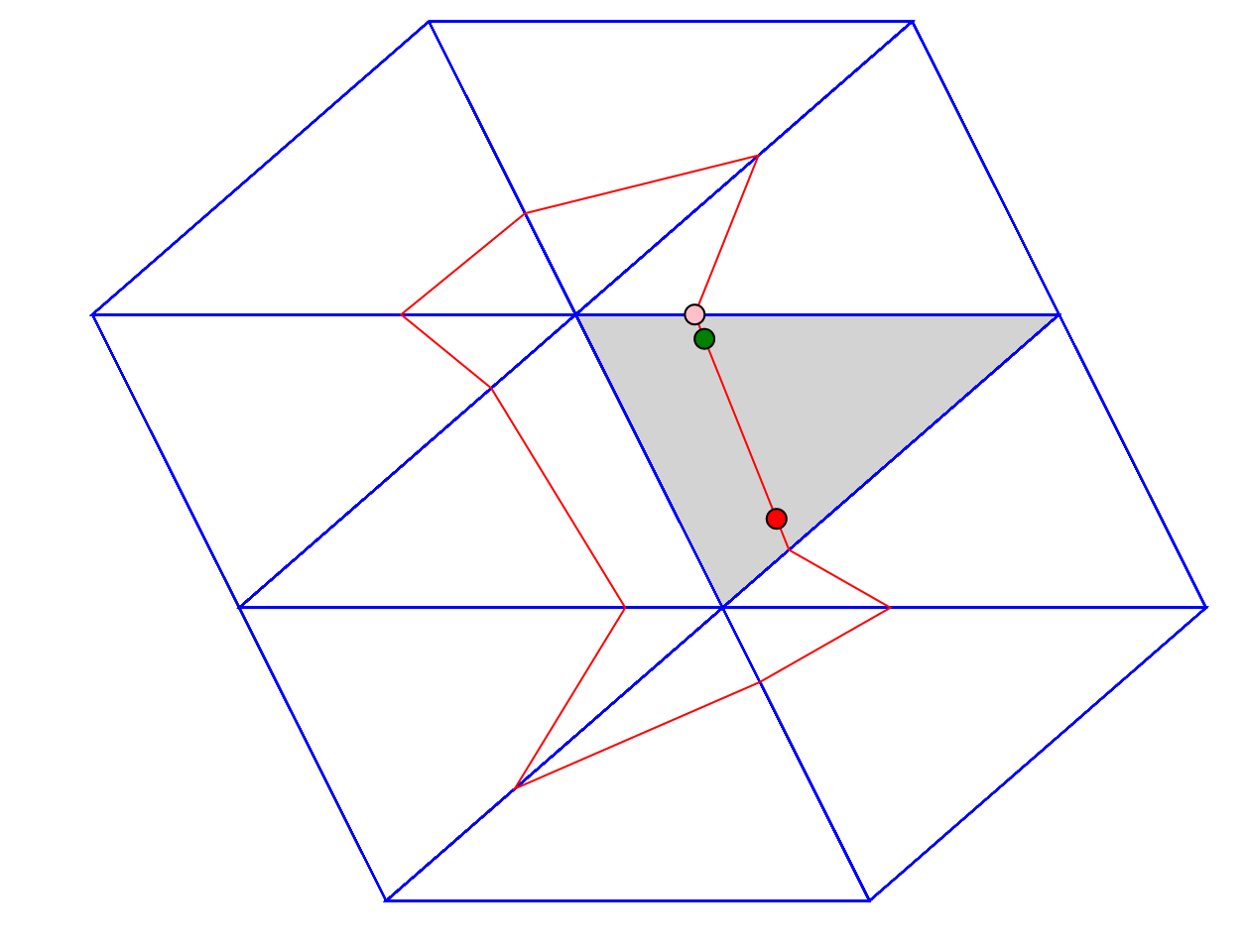}
\includegraphics[width=4cm]{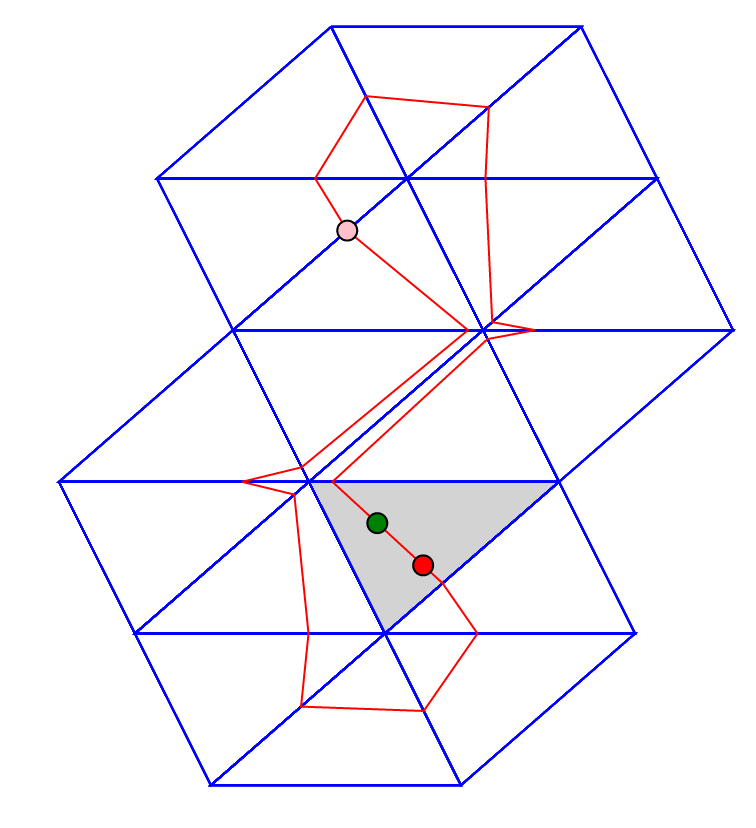}\\
\includegraphics[width=4cm]{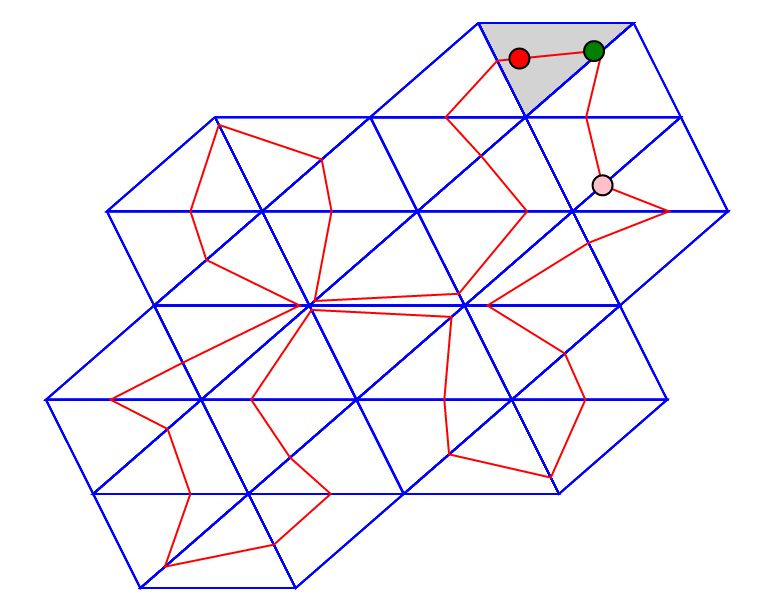}
\includegraphics[width=4cm]{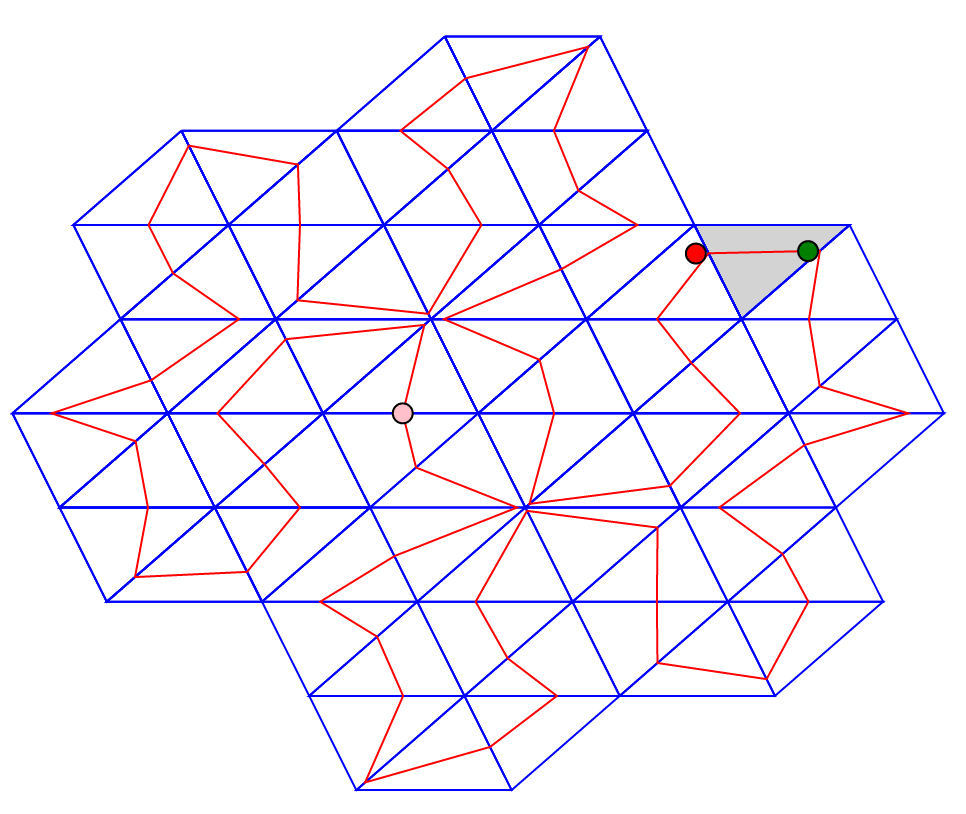}
\includegraphics[width=4cm]{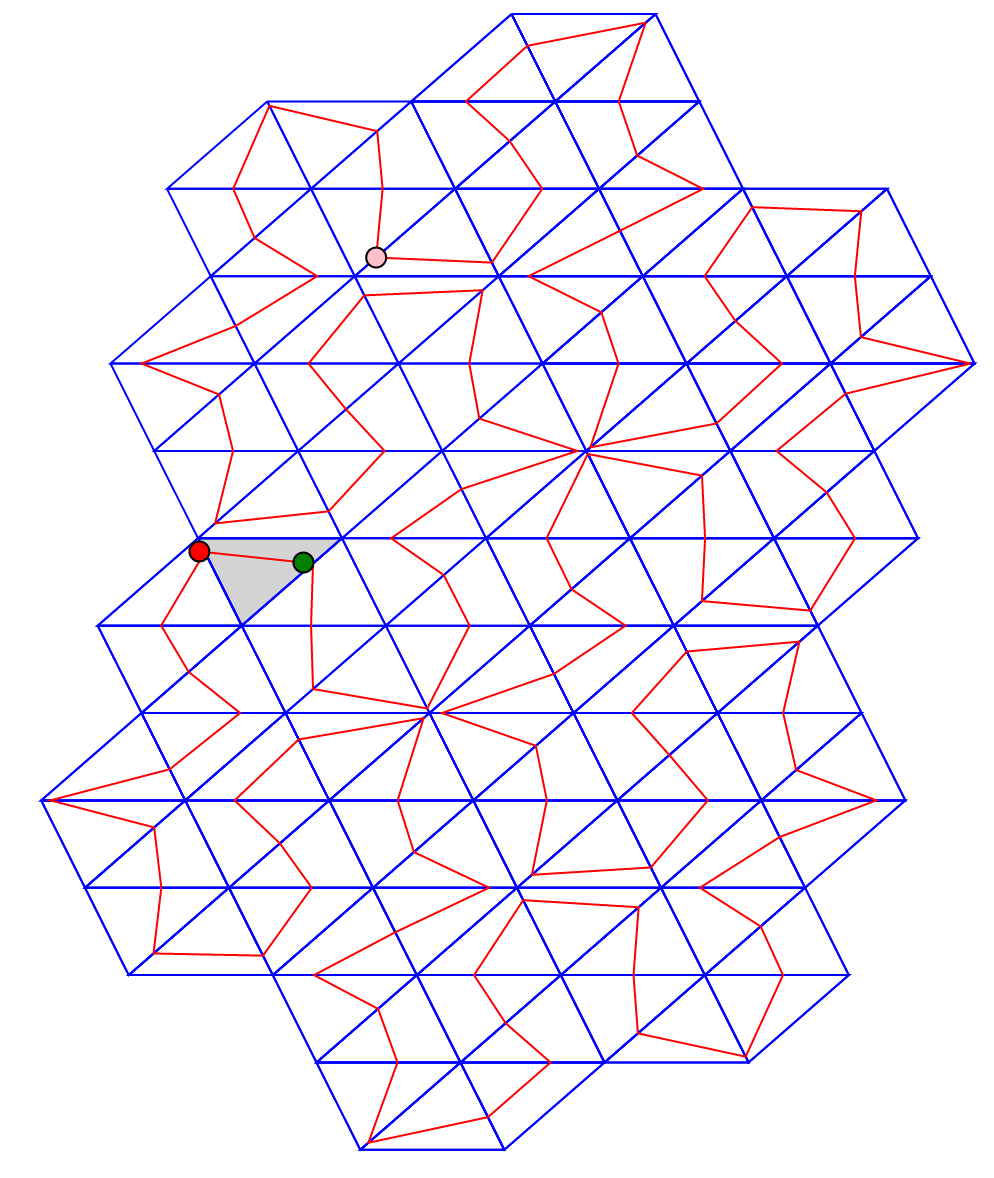}\\
\includegraphics[width=4cm]{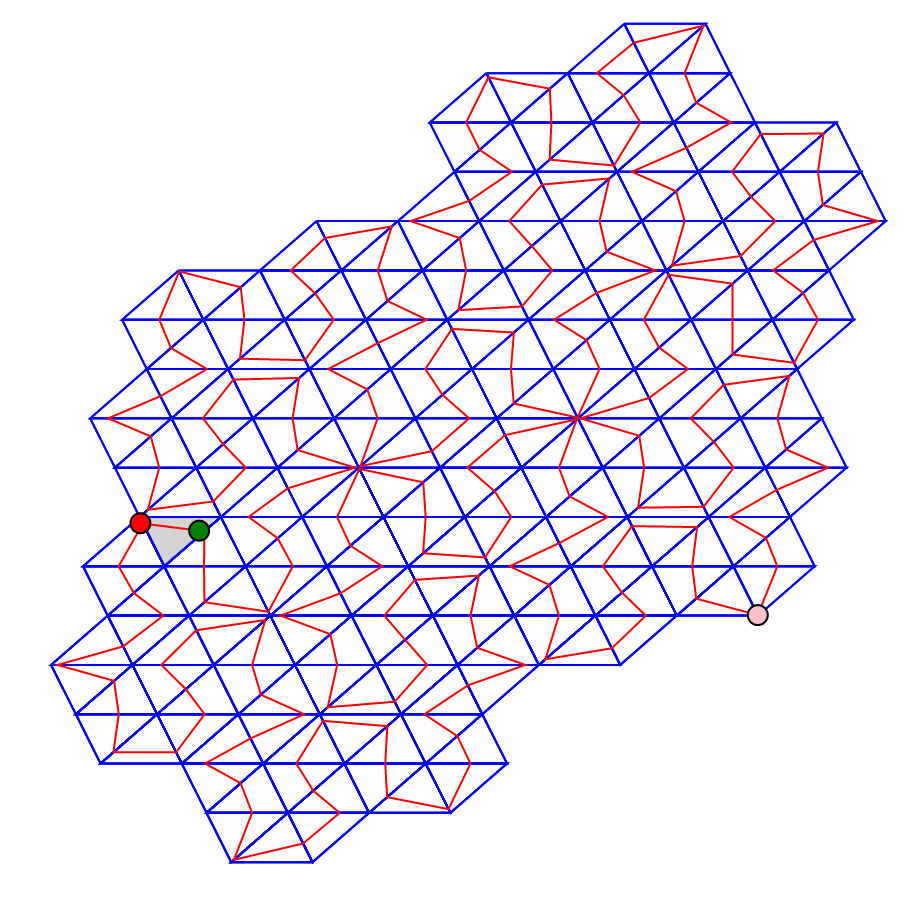}
\includegraphics[width=4cm]{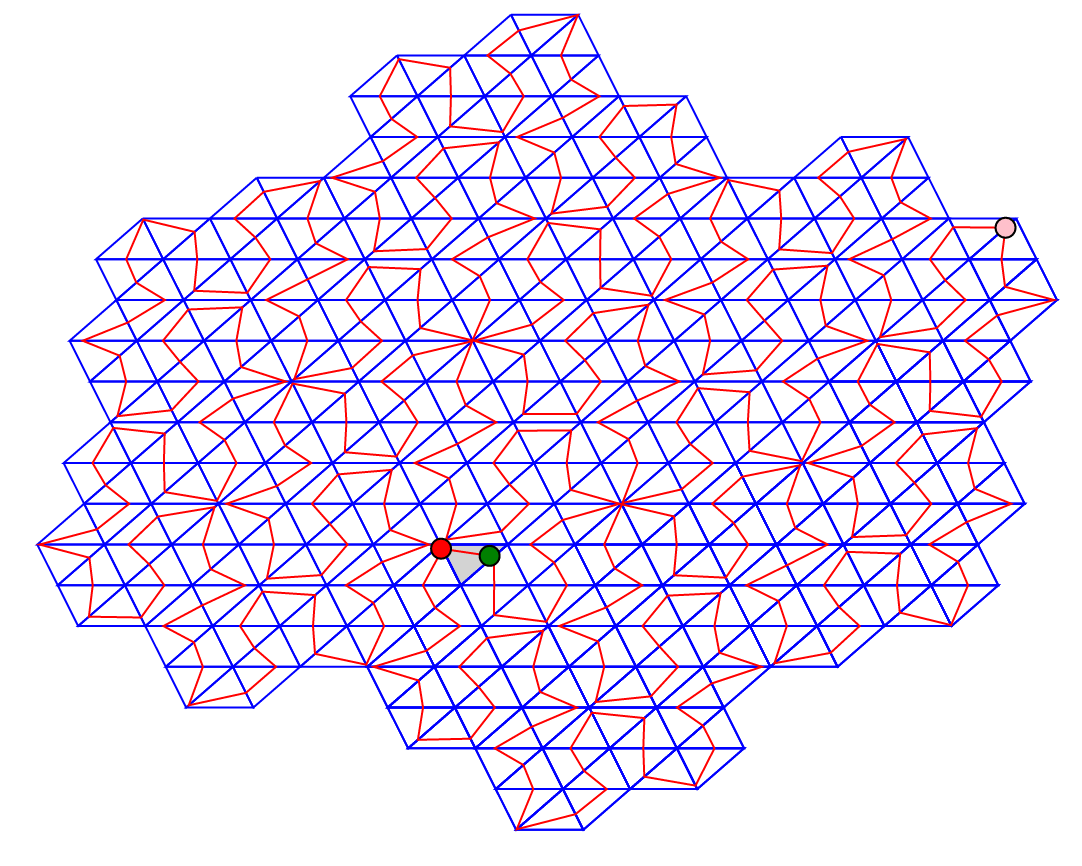}
\includegraphics[width=4cm]{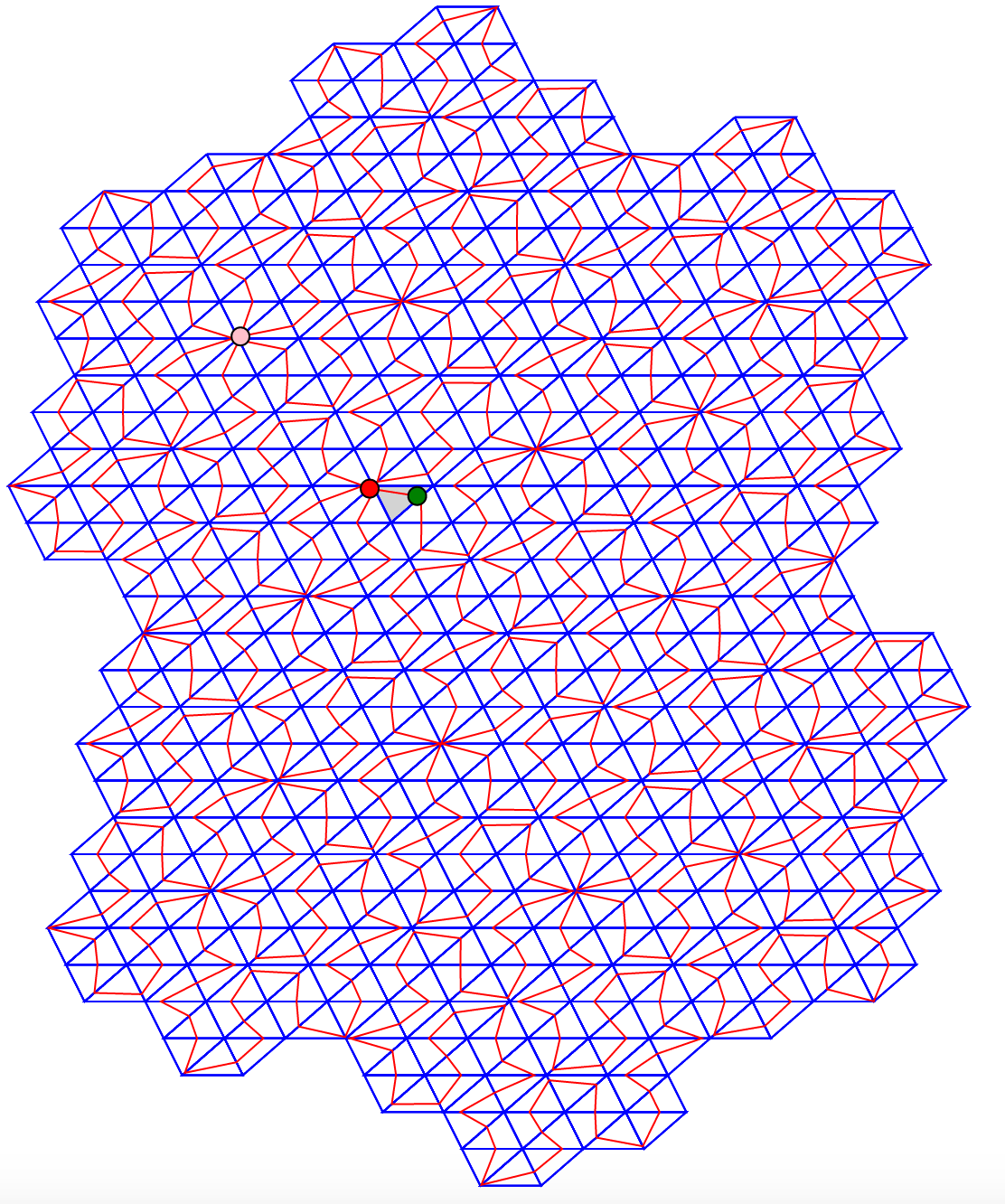}\\
\includegraphics[width=4cm]{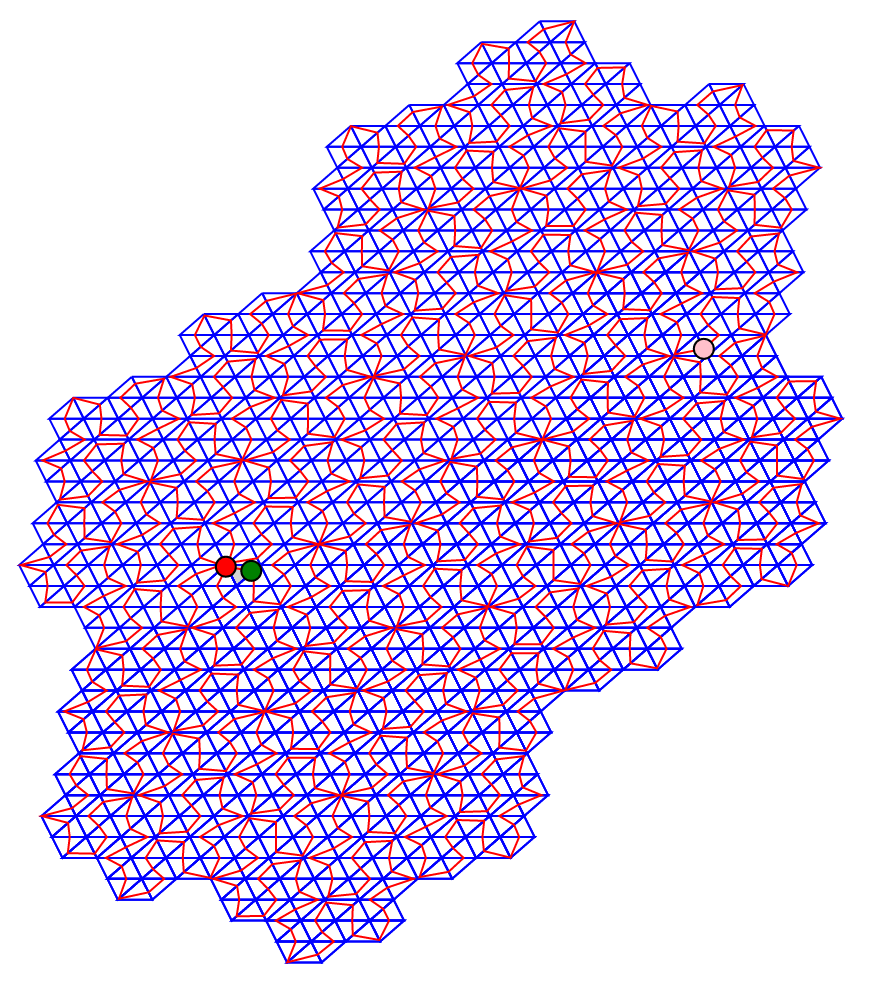}
\includegraphics[width=4cm]{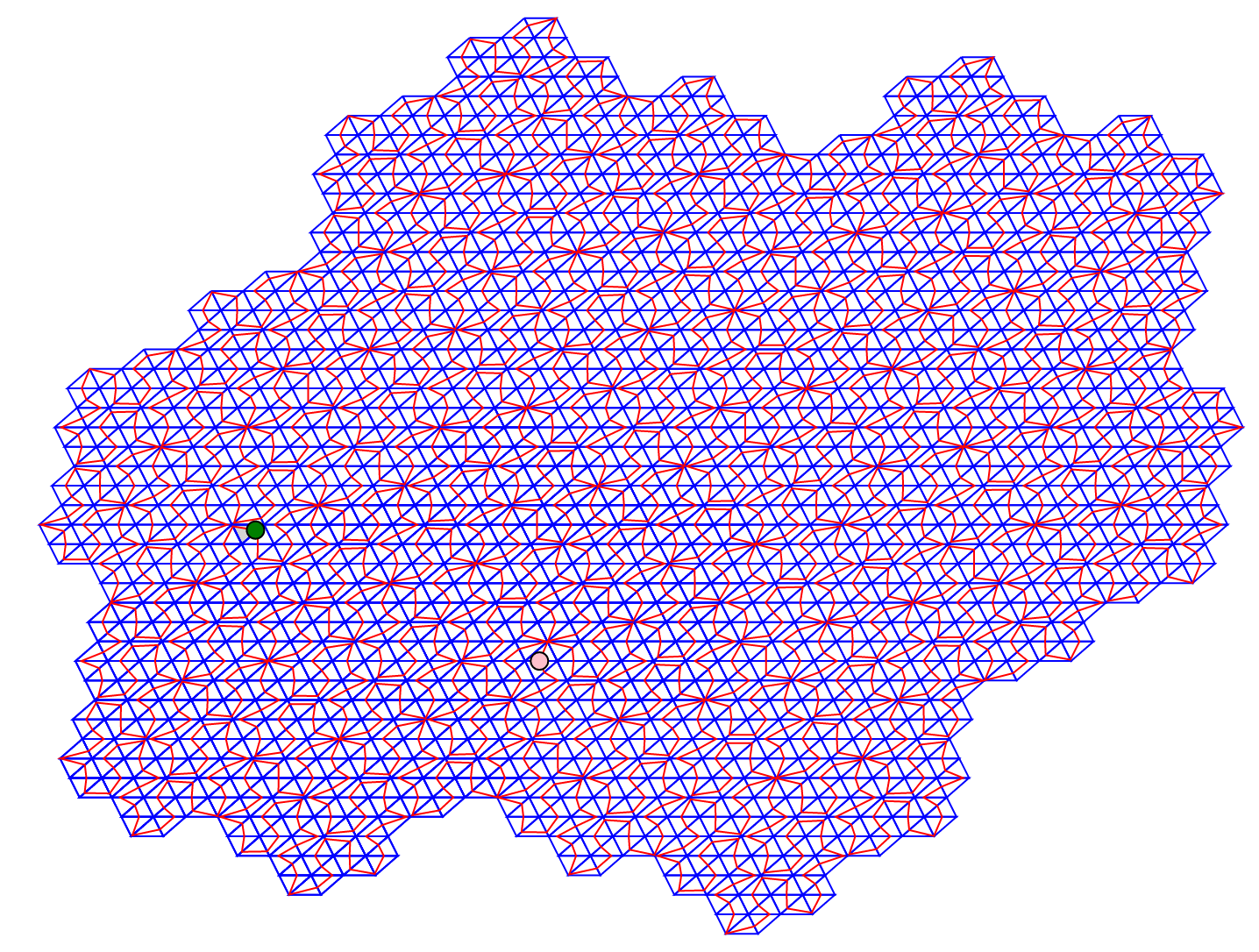}
\includegraphics[width=4cm]{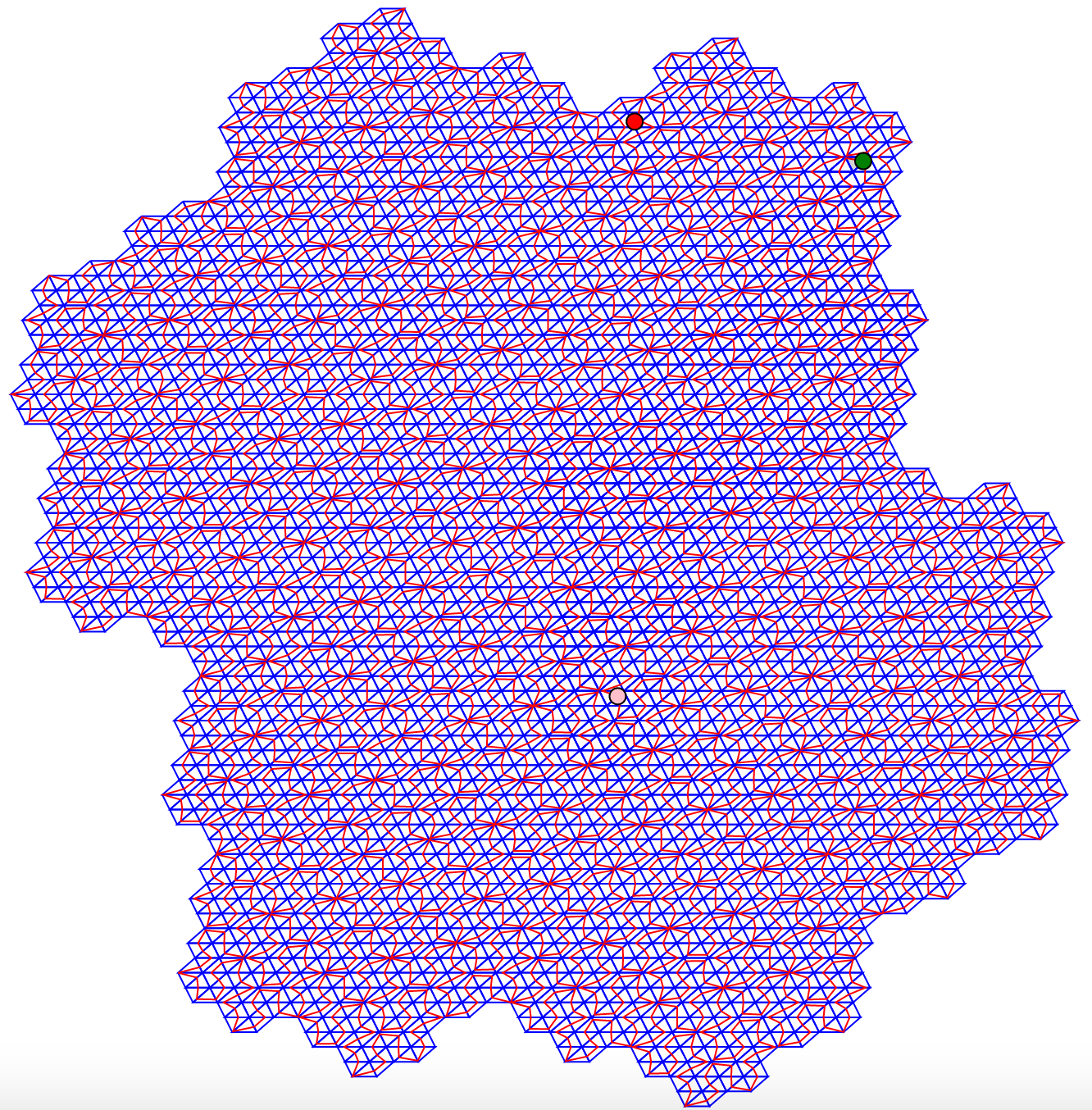}
\end{center}
\clearpage

\section{An example trajectory to cut and fold}\label{app:toy}

We have found it illuminating to cut and fold up trajectories on triangle tilings. Cut out along the outer edge of the picture, and fold along every edge of the tiling. The two short edges should be ``mountain folds'' and the long edges a ``valley fold.'' This will allow you to see for yourself the results of the basic Lemmas in \S\ref{sec:basic}. See Figure \ref{fig:hands} for guidance.

\vfill

\includegraphics[width=1.0\textwidth]{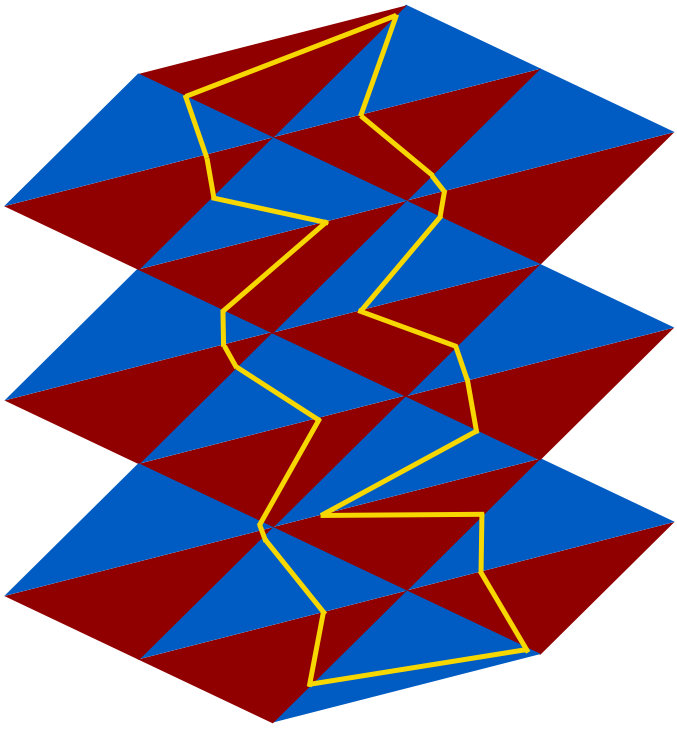}


\newpage
\begin{thebibliography}{9}

\bibitem[A88]{arnoux} Pierre Arnoux, \emph{Un exemple de semi-conjugaison entre un \'echange d'intervalles et une translation sur le tore}. Bulletin de la S. M. F., 116(4), 489--500, 1988.

\bibitem[AS13]{gasket} Pierre Arnoux, \u St\u ep\`an Starosta, \emph{The Rauzy gasket}, in \emph{Further Developments in Fractals and Related Fields}, Springer Science$+$Business Media, New York, 1--23 (2013).

\bibitem[DDRS18]{icerm} Diana Davis, Kelsey DiPietro, Jenny Rustad, Alexander St Laurent, \emph{Negative refraction and tiling billiards}, Advances in Geometry, 18(2), 133--159 (2018).

\bibitem[DH18]{trihex} Diana Davis, W. Patrick Hooper, \emph{Periodicity and ergodicity in the trihexagonal tiling}, to appear in Commentarii Mathematici Helvetici (2018).

\bibitem[G16]{glendinning} P.Glendinning, \emph{Geometry of refractions and reflections through a biperiodic medium}. SIAM Journal of Applied Mathematics, Vol. 76, No. 4, 1219--1238 (2016).

\bibitem[HP18]{olga} Pascal Hubert, Olga Paris-Romaskevich, \emph{Triangle tiling billiards draw fractals only if aimed at the circumcenter}, preprint (2018). arXiv:1804.00181.

\bibitem[HW18]{pat} W. Patrick Hooper, Barak Weiss, \emph{Rel leaves of the Arnoux-Yoccoz surfaces}, Selecta Mathematica, 24(2), 875--934 (2018).

\bibitem[LPV07]{cubic} J. H. Lowenstein, G. Poggiaspalla, and F. Vivaldi, \emph{Interval exchange transformations over algebraic number fields: the cubic Arnoux-Yoccoz model}, Dynamical Systems, 22(1), 73--106  (2007).

\bibitem[MF12]{MF} A. Mascarenhas, B. Fluegel: {\it Antisymmetry and the breakdown of Bloch's theorem for light}, preprint (c. 2012).

\bibitem[M12]{curt} Curtis T. McMullen, \emph{Cascades in the Dynamics of Measured Foliations}. Annales Scientifiques de l'\'Ecole Normale Sup\'erieure, 48(1), 1--39 (2012).

\bibitem[N89]{nonergodic} Arnaldo Nogueira, \emph{Almost all interval exchange transformations with flips are nonergodic}, Ergodic Theory $\&$ Dynamical Systems, 9(3), 15--25 (1989).

\bibitem[NPT13]{nogueiraflip} Arnaldo Nogueira, Benito Pires, Serge Troubetzkoy, \emph{Orbit structure of interval exchange transformations with flip}. Nonlinearity, IOP Publishing, 26 (2), 525--537 (2013).

\bibitem[SSS01]{shelby} R. A. Shelby, D. R. Smith, S. Schultz, {\it Experimental Verification of a Negative Index of Refraction}, Science, Vol. 292 no. 5514, 77--79 (2001).

\bibitem[SPW04]{smith}  D. Smith, J. Pendry, M. Wiltshire: {\it Metamaterials and negative refractive index}, Science, Vol. 305, 788--792 (2004).
\end{thebibliography}
\end{document}
